\documentclass[a4paper,11pt]{article}
\usepackage[utf8]{inputenc}
\usepackage[english]{babel}
\usepackage[margin=1.1in]{geometry} 
\usepackage[T1]{fontenc}
\usepackage{tgtermes}
\usepackage{amsthm}
\usepackage{amsmath}
\usepackage{bm}
\usepackage{enumitem}
\usepackage{amsfonts}
\usepackage{amssymb}
\usepackage{mathtools}
\usepackage{appendix}
\usepackage{mathrsfs}
\usepackage{setspace}
\usepackage{comment}
\usepackage{xcolor}
\usepackage{todonotes}
\usepackage{thmtools} 
\usepackage[
    pdfdisplaydoctitle,
    hypertexnames=false,
    colorlinks,
    breaklinks,
    urlcolor=blue,
    linkcolor=blue,
    citecolor=blue
]{hyperref}
\usepackage[nameinlink,capitalise]{cleveref}

\newcommand{\A}{\mathcal{A}}
\newcommand{\B}{\mathcal{B}}
\newcommand{\C}{\mathbb{C}}

\newcommand{\E}{\mathcal{E}}
\newcommand{\EE}{\mathbb{E}}
\newcommand{\F}{\mathcal{F}}
\newcommand{\G}{\mathcal{G}}

\renewcommand{\L}{\mathscr{L}}
\newcommand{\LL}{\mathcal{L}}
\newcommand{\M}{\mathcal{M}}
\newcommand{\N}{\mathbb{N}}

\newcommand{\PP}{\mathbb{P}}

\newcommand{\R}{\mathbb{R}}

\newcommand{\T}{\mathbb{T}}
\newcommand{\U}{\mathcal{U}}

\newcommand{\Z}{\mathbb{Z}}

\DeclareMathOperator{\id}{Id}

\let\div\relax

\newcommand{\div}{\operatorname{div}}

\renewcommand{\epsilon}{\varepsilon}
\renewcommand{\setminus}{\smallsetminus}
\newcommand{\eps}{\epsilon}
\renewcommand{\det}{{\operatorname{det}}}

\newcommand{\one}{\bm{1}}

\newcommand{\norm}[1]{\left\|#1\right\|}
\newcommand{\brak}[1]{\left\langle#1\right\rangle}
\newcommand{\dd}{\mathrm{d}}

\newcommand{\Do}{\mathsf{D}}

\renewcommand{\emptyset}{\varnothing}

\newcommand{\expt}[2][]{\mathbb{E}_{#1}\left[#2\right]}

\newtheorem{theorem}{Theorem}[section]
\newtheorem{definition}[theorem]{Definition}
\newtheorem{hypothesis}[theorem]{Hypothesis}
\newtheorem{corollary}[theorem]{Corollary}
\newtheorem{lemma}[theorem]{Lemma}
\newtheorem{proposition}[theorem]{Proposition}

\theoremstyle{remark}
\newtheorem{remark}[theorem]{Remark}

\numberwithin{equation}{section}

\author{
  Antonio Agresti\thanks{Department of Mathematics Guido Castelnuovo, 
  Sapienza University of Rome, P.le Aldo Moro 5, 00185 Rome, Italy.
  \texttt{antonio.agresti@uniroma1.it}}
  \and
  Federico Butori\thanks{Scuola Normale Superiore, Piazza dei Cavalieri 7, Pisa, Italy.
  \texttt{federico.butori@sns.it}}
  \and
  Eliseo Luongo\thanks{Fakultät für Mathematik, Universität Bielefeld, 33501 Bielefeld, Germany.
  \texttt{eluongo@math.uni-bielefeld.de}\newline
    Current Address: Dipartimento di Matematica, Politecnico di Milano, Via E. Bonardi 9, 20133 Milano, Italy}
}
\title{Global smooth solutions by high mode Lie-Transport noise for Logarithmically Hyperdissipative Navier-Stokes equations}

\allowdisplaybreaks

\begin{document}
\maketitle
\begin{abstract}
We study a logarithmically hyperviscous Navier–Stokes model on the three-dimensional torus with Lie-transport noise, which includes both transport and stretching. We prove that, for noise of sufficiently large intensity and high frequency, the system admits a unique global smooth solution with probability arbitrarily close to one. Unlike previous works, this physically motivated noise does not preserve energy or enstrophy, but rather circulation. Moreover, an enhanced dissipation result for the velocity is also obtained. These results are established through a probabilistic mechanism that produces effective dissipation via a scaling limit. Crucially, our approach bypasses the lack of conserved quantities and tames the singular nature of stochastic stretching. 
\end{abstract}

\tableofcontents

\bigskip

\section{Introduction}\label{sec:intro}
Starting from the seminal paper \cite{galeati_convergence_2020}, it is nowadays well-known that high-mode, transport-type noise has enhancing dissipating properties, strongly related to its mixing and enhanced dissipation features; see e.g., \cite{flandoli_quantitative_2024,HMPZZ25, LuoMixing} and the references therein. This led in \cite{flandoli_high_2021} to the suppression of potential vorticity blow-up for the 3D stochastic Navier-Stokes equations with transport noise at the vorticity formulation, namely the system 
\begin{equation}\label{eq_flandoli_luo}\tag{SNSTN}
    \begin{cases}
      \mathrm{d}\omega+{\Pi}\left(\circ\, \mathrm{d}W \cdot \nabla \omega\right)=\left(\nu \Delta \omega-u\cdot\nabla \omega+\omega\cdot\nabla u\right)\mathrm{d}t,\\
       u=K[\omega], 
    \end{cases}
\end{equation}
with high probability\footnote{The symbol $K$ stands for the Biot-Savart operator, while $\Pi$ corresponds to the Leray projection, see Subsection \ref{ss:notation}. The latter has to be included in the system above, since in principle the term $W \cdot \nabla \omega$ is not divergence-free, contrary to the others.}. In the above, the noise $W$ is required to have sufficiently large intensity and sufficiently high spectrum. While conceptually striking, this result relies on a noise structure that is not physically motivated: it preserves enstrophy and does not capture vortex stretching, which is intrinsic to three-dimensional fluid dynamics; see the foundational work \cite{holm2015variational}, as well as the discussions in \cite{cotter2020data,crisan2023implementation}, \cite[Section 1.2]{flandoli_high_2021}, and \cite[Chapter 5]{flandoli_stochastic_2023}. 
Extending the results of \cite{flandoli_high_2021} to the physically motivated stochastic model of \cite{holm2015variational}:
\begin{equation}\label{eq_holm}\tag{SNS}
   \begin{cases}
      \mathrm{d}\omega+\circ\, \mathrm{d} W \cdot \nabla \omega-\omega\cdot \circ \,\mathrm{d} \nabla W=\left(\nu \Delta \omega-u\cdot\nabla \omega+\omega\cdot\nabla u\right)\mathrm{d}t,\\
       u=K[\omega],
   \end{cases}
\end{equation}
 and its hyperviscous variants has proven to be extremely challenging. Existing regularization by noise results mainly concern noise structures that preserve key quantities (e.g., energy or enstrophy) in fluid dynamics, thereby leading to a more tractable analytical setting. In this paper, we prove, for the first time, regularization by noise for a logarithmically hyperviscous Navier–Stokes model with transport and stretching noise, also called Lie-Transport noise, having the same scaling as the 3D Navier-Stokes equations \eqref{eq_holm}. Before proceeding with the presentation of our contribution, we briefly discuss the challenges that Lie-Transport noise poses.

\smallskip 
 
The difficulties in treating systems like \eqref{eq_holm} underscore a recurring theme: while noise can have a stabilizing effect, like in \eqref{eq_flandoli_luo}, its precise structure and interaction with nonlinear mechanisms are decisive, and physically realistic perturbations may either regularize or destabilize the dynamics. Indeed, \cite[Appendix~2]{flandoli_high_2021} shows that the stochastic stretching term seems to become singular under the scaling limit of \cite{galeati_convergence_2020}. This also happens for linear variants of \eqref{eq_holm}; see, for example, \cite[Chapter 3]{flandoli_stochastic_2023}, and seems related to the fact that high-intensity isotropic vector fields, like the ones considered in the scaling limit of \cite{galeati_convergence_2020}, can produce negative eddy viscosity phenomena, cf. \cite{frisch1991, gama1994negative,lanotte1999large}.

In \cite{agresti_global_2024}, the first author took a step forward in understanding the stabilizing properties of transport-stretching noise by establishing global well-posedness with high probability for the following hyperviscous Navier-Stokes equations with transport noise:
\begin{align}\label{eq_antonio}\tag{SHNS}
\begin{cases}
\mathrm{d} u+\left[ (-\Delta)^{\gamma}u+\Pi(u\cdot\nabla u)\right]\mathrm{d}t=\Pi\left(\circ \mathrm{d}{W}\cdot \nabla u\right),\\
\operatorname{div}u=0,
\end{cases}
\end{align}
where $\gamma>1$ is arbitrary. In contrast to \eqref{eq_flandoli_luo}, physical motivations for transport noise at the level of the velocity equation can be found in, e.g., \cite[Subsection 1.2]{agresti_anomalous_2024} and \cite{debussche_second_2024,MR04_turbulent,memin2014fluid}. 
The reader is also referred to the recent paper \cite{jiao2026delayed}, where fractional operators appear in the driving noise, rather than in the differential operator in the deterministic part of the equation, and to \cite{Zhao_delayed26}, where the delayed blow-up of \eqref{eq_flandoli_luo} is considered in the case of Navier-slip boundary conditions.

It is worth pointing out that, besides the hyperviscosity, \eqref{eq_antonio} does not correspond to either \eqref{eq_holm} or \eqref{eq_flandoli_luo}. 
In particular, the stochastic stretching terms appearing in \eqref{eq_antonio} are stronger compared to \eqref{eq_flandoli_luo} but weaker than those in \eqref{eq_holm}. 
Let us also mention that the noise considered in \eqref{eq_flandoli_luo} is enstrophy-preserving, while, as already mentioned, the one in \eqref{eq_antonio} is energy-preserving, {and that the one considered in \cite{jiao2026delayed} preserves an appropriate Sobolev norm}. On the other hand, the one in \eqref{eq_holm} conserves neither, being circulation-preserving. This prevents the possibility of extending deterministic a priori bounds to the stochastic system \eqref{eq_holm}. In view of the comments above, understanding the stabilizing mechanism of the physically motivated stochastic perturbations introduced in \cite{holm2015variational} is completely open and highly non-trivial. Only recently were some steps forward in the understanding of their effects (on linear models) made by the second and third named authors, with collaborators, in the series of works \cite{butori2026ito, butori2026background, butori_meanfield_2025}. 

\smallskip

In this manuscript, we prove global well-posedness with large probability of \eqref{eq_holm} with a slight hyperviscosity weaker than the one present in \eqref{eq_antonio}; see \eqref{main_strat} below. 
To give an informal version of our main results, we introduce some notation.  
Let $\T^3=\R^3/\Z^3$ be the three-dimensional torus, and let $A=\Delta \log(1+(-\Delta)^{1/2})$ be the operator with Fourier multiplier $-|k|^2\log(1+|k|)$, i.e.,
\begin{align}\label{logaritmicA}
    \mathcal{F}\left[A\phi\right](k)=-|k|^2\log(1+|k|)\mathcal{F}\left[\phi\right](k)\quad \text{for }k\in 2\pi\Z^3,
\end{align}
which acts componentwise on vector fields.
For the results in this work, the logarithmic growth of the symbol is not essential. Indeed, it is sufficient that $A$ is a Fourier multiplier with symbol growing faster than $|k|^2$ corresponding to the Laplace operator; see \autoref{hp_correction} and \autoref{Defintition_operator} below.
Consider the following stochastic Navier-Stokes equations with hyperviscosity and Lie noise:
\begin{equation}\label{main_strat}\tag{SLNS}
\begin{cases}
    \mathrm{d}\omega + \left( \LL_u \omega - A\omega\right)\mathrm{d}t + \sqrt{\nu}\sum_{k\in I}\LL_{\sigma_k} \omega \circ \mathrm{d}W_t^k =0 \quad &\text{on }\T^3,\\
    u= K[\omega]\quad &\text{on }\T^3,\\
    \omega|_{t=0}=\omega_0,
    \end{cases}
\end{equation}
where $\omega$ is the vorticity field of an incompressible fluid and $u$ denotes the corresponding velocity field. Given two vector fields $X, Y$, the notation $\LL_X Y$ stands for the Lie derivative, i.e., 
\begin{align*}
\LL_X Y\coloneq
X\cdot\nabla Y-Y\cdot\nabla X.
\end{align*}
Moreover, $\nu>0$ is the noise intensity and, given the set of indices 
$
    I\coloneq2\pi \Z^3_0\times \{1,2\},
$
for $k=(m,j),$ the noise coefficients $\sigma_k$ satisfy
\begin{align*}
    \sigma_k(x)=e^{im\cdot x}\theta_m a_{m,j},
\end{align*}
where $\theta=\{\theta_m\}_{m\in 2\pi\Z^3_0}\in \ell^2 $ and, for each $m\in 2\pi \Z^3_0$, $\{\frac{m}{|m|},\ a_{m,1},a_{m,2}\}$ is an orthonormal system of $\R^3$.
Finally, $\circ$ denotes the Stratonovich product and $\{W^k\}_{k\in I}$ is a family of complex
Brownian motions. Further assumptions on the noise coefficients $\sigma_k$'s and the complex Brownian motions $W^{k}$'s are described in detail below, cf. \autoref{sec_ito}. Here we only anticipate that the value of $\theta_m$ depends only on $|m|$ and, in view of this restriction, the noise in \eqref{main_strat} can be reformulated using only real-valued objects. However, we employ the complex formulation, as it is more convenient for computations.  

Our first main result shows that despite the presence of strong stochastic vortex stretching and the lack of conserved quantities, the noise still ensures the existence of global smooth solutions to \eqref{main_strat} with probability arbitrarily close to one. Namely, the following holds:

\begin{theorem}[Global smooth solutions of LNSEs by transport-stretching noise – Informal statement]\label{informal_thm1}
    Let $ M>0$ and $ \epsilon\in (0,1)$. There exist $\nu>0,\ \theta\in \ell^2(2\pi\Z^3_0)$ such that for every zero-mean, divergence-free vector field $\omega_0$ satisfying
    \begin{align}
    \label{eq:bound_initial_data_intro}
        \|\omega_0\|_{L^2(\T^3;\R^3)}\leq M,
    \end{align}
    there exists a unique smooth solution to \eqref{main_strat} that is global-in-time with probability larger than $1-\epsilon$.
\end{theorem}

The above is an informal and suboptimal version of \autoref{thm:main} below. 
In particular, in the latter, we can allow a larger class of (subcritical) initial data. We refer to \autoref{sec:setup} for the definition of all the notation involved in the quoted theorem.

The mechanism behind the validity of the above result lies in the possibility of extracting an eddy dissipation mechanism (or, in other words, enhanced dissipation) from the transport stretching noise, similarly to the quoted literature, cf. \cite{galeati_convergence_2020, flandoli_quantitative_2024}. However, high-intensity isotropic vector fields can produce negative eddy viscosity phenomena, cf. \cite{frisch1991,lanotte1999large}. We exploit the presence of the hyperdissipative operators to dominate these effects.  

The following result complements \autoref{informal_thm1} and captures the enhanced dissipation effect of Lie-Transport noise at the level of velocity fields. 
 
\begin{theorem}[Enhanced dissipation of velocity for LNSEs by transport-stretching noise – Informal statement]\label{informal_thm2}
    Let $M,\chi,b>0$, and $ \epsilon\in (0,1)$.  There exist $\nu>0,\ \theta\in \ell^2(2\pi\Z^3_0)$ such that for every zero-mean, divergence-free vector field $\omega_0$ satisfying \eqref{eq:bound_initial_data_intro}, the unique smooth solution to \eqref{main_strat} satisfies
   \begin{align*}
       \|\omega_t\|_{H^{-1}(\T^3;\R^3)}\leq D e^{-\chi t}\|\omega_0\|_{H^{-1}(\T^3;\R^3)} \ \text{ a.s.\ for all }t<\mu,
   \end{align*}
    where $\mu$ is a stopping time and $D$ is a random constant with 
    \begin{align*}
        \PP(\mu=\infty)>1-\varepsilon\qquad \text{ and } \qquad \mathbb{E}[D^b] \lesssim_{M,\varepsilon,\chi,b} 1.
    \end{align*}
\end{theorem}
The above is a special case of \autoref{thm_global} below. Moreover, in view of the relation between the vorticity $\omega$ and the velocity field $u= K[\omega]$, it holds
\begin{equation}
\label{eq:velocity_vorticity_intro}
 \|u\|_{L^2(\T^3;\R^3)}\eqsim \|\omega\|_{H^{-1}(\T^3;\R^3)}.
\end{equation}
Thus, the previous result shows that transport-stretching noise ensures a much faster decay of energy than in the case of pure diffusion.
The interpretation of \autoref{informal_thm2} is related to the physical role of turbulent mixing in enhancing effective diffusion \cite{batchelor1959,dimotakis2005}.
In fluid models, mixing by unresolved turbulent motions affects
energy transfer at large scales. An effective eddy viscosity is often used to account for this effect in the equations for the resolved flow \cite{kraichnan1976eddy}. These motions can also be modelled by noise acting on the resolved
flow \cite{debussche_second_2024,flandoli_additive_2022}.
From this perspective, \autoref{informal_thm2} shows that small-scale fluctuations can enhance the dissipation of the resolved velocity field even in the presence of stochastic vortex stretching.

\subsection{Blow-up in 3D Fluid Models and Regularization by Noise}
\label{ss:blow_up_3D_models}
The competition between vortex stretching and dissipation is a basic problem in the regularity theory of 3D fluid equations. Hyperviscous variants of the Navier-Stokes equations have received considerable attention as a framework for investigating
this competition, since they retain several structural features of the classical Navier-Stokes equations while strengthening dissipation.
In velocity formulation, they read
\begin{align}\label{eq_det_hyper}
    \partial_t u + u\cdot\nabla u + \nabla p - Au = 0,
\end{align}
where $A$ is as defined in \eqref{logaritmicA} or $A=-(-\Delta)^\gamma$ with $\gamma>1$.
Although they can be viewed as a regularized version of the latter due to the presence of hyperviscosity, similarly to the Navier-Stokes case, the only a priori estimate available in general up to the (possible) blow-up time is the energy balance.

As shown by J.-L. Lions \cite{lions_global}, if $A = -(-\Delta)^{\gamma}$ with $\gamma \geq \frac{5}{4}$, the energy balance
\begin{align}\label{energy_eq}
\frac{1}{2}    \|u(t)\|^2_{L^2(\T^3;\R^3)} + \int_0^t \|(-A)^{1/2}u(s)\|_{L^2(\T^3;\R^3)}^2\, \mathrm{d}s \leq 
\frac{1}{2}\|u_0\|^2_{L^2(\T^3;\R^3)}
\end{align}
is sufficient to prevent blow-up of strong solutions. Scaling considerations help to understand the relevance of this threshold. In particular, in the case $A = -(-\Delta)^{\gamma}$, the space $L^\infty(\R_+; H^{\frac{5}{2}-2\gamma}(\T^3))$ is (locally) invariant under the self-similar scaling of \eqref{eq_det_hyper}: $u\mapsto u_\lambda$ for $\lambda>0$, where
\begin{align*}
    u_{\lambda}(t,x)=\lambda^{1-\frac{1}{2\gamma}}u(\lambda t, \lambda^{\frac{1}{2\gamma}}x),
\end{align*}
see \cite[Subsection 1.2]{agresti_global_2024} for details.
Therefore, the energy space $L^2(\T^3;\R^3)$ is critical (respectively subcritical) for the hyperviscous Navier-Stokes equations in velocity formulation if and only if $\gamma = \frac{5}{4}$ (respectively $\gamma > \frac{5}{4}$). This threshold is expected to play a role both for the possible blow-up of solutions and for the nonuniqueness of weak solutions satisfying the energy inequality \eqref{energy_eq}; see, for instance, \cite{albritton_nonuniqueness_2022, jia_are_2015}. 

In terms of the vorticity formulation of the Navier-Stokes equations, such as \eqref{main_strat} considered here, the critical space is shifted by one degree of spatial smoothness, and it is given by (in the $L^2$-scale)
\begin{equation}
\label{eq:critical_smoothness_vorticity}
H^{\frac{3}{2}-2\gamma}(\T^3;\R^3).
\end{equation}
Consequently, all spaces of the form $L^\infty_t H^{s}_x$ with $s > \frac{3}{2}-2\gamma$ are subcritical for the equation in vorticity formulation. Let us also remark that, with respect to scaling and criticality, the operator $A$ defined in \eqref{logaritmicA} enjoys the same properties as $-\Delta$ in the classical Navier-Stokes equations.

Several attempts have been made to extend the global well-posedness result of \cite{lions_global} to operators that are less regularizing than $-(-\Delta)^{5/4}$. Tao \cite{tao2010global}, and Barbato, Morandin, and Romito \cite{barbato2015global} proved global well-posedness when $-(-\Delta)^{5/4}$ is slightly weakened by a suitable Fourier multiplier. Later, Colombo, De Lellis, and Massaccesi \cite{colombo2020generalized} showed that there exists $\varepsilon>0$, depending only on the size of the initial condition in suitable function spaces, such that \eqref{eq_det_hyper} admits a global smooth solution whenever $\gamma \geq \frac{5}{4}-\varepsilon$. The logarithmically hyperviscous operator \eqref{logaritmicA} considered here lies well below these regimes, and global regularity for arbitrary smooth initial data is unknown. The gap between the available energy estimate and the bounds needed to ensure global regularity is substantial and is probably not merely technical, reflecting the supercritical nature of the energy balance \eqref{energy_eq} with respect to the natural scaling of \eqref{eq_det_hyper}. This perspective is reinforced by the recent groundbreaking works \cite{alpoge_buckmaster2026blowup, openai2026blowup} on finite-time blow-up of smooth solutions for 3D Navier-Stokes equations in the unforced and forced settings, respectively, addressing one of the Millennium Prize Problems \cite{fefferman2006existence}.

These developments provide further motivation for understanding whether stochastic perturbations can suppress singularity formation in fluid models: if blow-up of smooth solutions occurs only along exceptional, \emph{non-generic} trajectories, stochastic perturbations may drive the system away from such trajectories, thereby restoring well-posedness. The broader idea that stochastic perturbations can improve well-posedness goes back to the seminal works of Zvonkin and Veretennikov \cite{veretennikov1981strong, zvonkin1974transformation} and has since been developed in a wide range of finite- and infinite-dimensional settings (see, for instance, \cite{ bagnara2025suitable,bagnara2025no,
delarue2014noise, Vicol_reg, herr2023three,hong2024stochastic, lang2025global} for examples of this phenomenon preventing blow-up in finite and infinite-dimensional systems).

In the context of fluid dynamics models, it is now well understood that transport-type noise can improve the well-posedness theory, both by preventing blow-up of smooth solutions \cite{fedrizzi2013noise, flandoli_high_2021} and by restoring uniqueness in low-regularity classes \cite{coghi_existence_2023, flandoli_wellposedness_2010}. 
For the 3D Navier-Stokes equations with transport noise in velocity formulation, the first author established the existence of global smooth solutions with high probability for \eqref{eq_antonio} whenever $\gamma>1$ in \cite{agresti_global_2024}. More recently, after the appearance of the first version of this paper, he also established global smoothness with high probability for \eqref{eq_antonio} in the classical case $\gamma=1$
in \cite{A26_absence}. Removing the hyperviscosity required a different approach, combining homogenization techniques with the energy-preserving structure of the noise. As already mentioned, obtaining similar results in case of Lie-transport noise presents additional difficulties: this noise incorporates stochastic vortex stretching, an important feature of three-dimensional fluid models \cite{flandoli_stochastic_2023,holm2015variational}, and, in view of this, preserves neither energy nor enstrophy. In particular, the energy conservation exploited in \cite{A26_absence} is unavailable in this setting; see \autoref{r:lack_energy} below for more details.

\smallskip

\autoref{informal_thm1} and \autoref{informal_thm2} provide, to the best of our knowledge, the first example in which the regularizing properties of transport-stretching noise are investigated. In particular, they show that, with large probability, blow-up can be suppressed and enhanced dissipation of the velocity holds for a fluid model having the same scaling properties and critical spaces as the three-dimensional Navier-Stokes equations. As outlined in the next subsection, the proof identifies an effective dissipative mechanism generated by small-scale fluctuations, while accounting for possible negative eddy viscosity effects at the level of the velocity energy \cite{frisch1991,gama1994negative,lanotte1999large}.

\subsection{Strategy of the proof}

\subsubsection{Global smooth solutions by transport-stretching noise (\autoref{thm:main})}
\label{sss:strategy_main}
Inspired by the work \cite{flandoli_high_2021}, the main idea behind \autoref{thm:main} is to exploit the scaling limit developed in \cite{butori_meanfield_2025} on a fixed time horizon $[0,T]$ for $T\gg 1$, 
together with a gluing procedure at time $T$. More precisely, for noise coefficients $(\sigma_k^n)_n$ satisfying suitable isotropy assumptions and concentrating at smaller and smaller spatial scales as $n \to \infty$ (see \autoref{sec_ito}), the solution $\omega^n$ to the truncated version of the stochastic 3D Navier-Stokes equations with transport-stretching noise\footnote{Here we write everything in It\^o form (see \autoref{sec_ito} for comments), making explicit the structure of the It\^o-Stratonovich corrector which plays a major role in the argument.}
\begin{align}\label{main_ito_cutoff_scaling}
 \begin{cases}
    \mathrm{d}\omega^n + \big[ \phi_{R}(\|\omega^n\|_{H^r_g})\LL_{u^n} \omega^n -A\omega^n-\nu \Delta\omega^n\big]\,\mathrm{d}t + \sqrt{\nu}\sum_{k\in I}\LL_{\sigma^n_k} \omega^n\,\mathrm{d}W_t^k =0 \\
    u^n= K[\omega^n],
    \end{cases}   
\end{align}
converges to the solution $\overline{\omega}$ of the deterministic equation
\begin{align}\label{main_det_cutoff_intro}
 \begin{cases}
    \partial_t\overline{\omega} +  \phi_{R}(\|\overline{\omega}\|_{H^r_g})\LL_{\overline{u} }\overline{\omega} -A\overline{\omega} - \nu\Delta \overline{\omega} =0 \\
    \overline{u}= K[\overline{\omega}],
    \end{cases}   
\end{align}
with the same initial data, where $H^r_g\subseteq H^r$ (see \autoref{def_function_spaces} for the precise formulation), $A=\Delta g(\sqrt{-\Delta})$ with $g(r)\to \infty$ as $r\to \infty$ (\autoref{hp_correction} and \autoref{Defintition_operator}), $R>0$ and 
\begin{align*}
    \phi_{R}(\rho)=
    \left\{
    \begin{aligned}
        &1&\quad & \text{ if } \rho\leq R\\
        &0& \quad &\text{ if } \rho> R+1.
    \end{aligned}
\right.
\end{align*}
We emphasize that the condition $r>-1/2$ encodes the \emph{subcriticality} of the space $H^{r}_g$ for the 3D Navier-Stokes equations. The latter condition is essential to obtain global well-posedness for the truncated problem \eqref{main_det_cutoff_intro} as shown by the scaling argument in \autoref{ss:blow_up_3D_models}.

\smallskip

If
\begin{equation}
\label{eq:convergence_prob_intro}
\omega^n \to \overline{\omega}
 \ \text{ in probability in }\ C([0,T];H^{r}_g),
\end{equation}
then \autoref{thm:main} follows by exploiting the dissipative structure of the limiting equation. Indeed, if $\nu$ is chosen sufficiently large with respect to the initial condition of \eqref{main_det_cutoff_intro}, then the Laplacian term enforces exponential decay of $\|\overline{\omega}(t)\|_{H^r_g}$, uniformly in $R$. In particular, for $R$ large enough, one has $\|\overline{\omega}(t)\|_{H^r_g} \leq R-1$ for all $t \in [0,T]$. If $T$ is sufficiently large, $\overline{\omega}_T$ is arbitrarily small. Combining this decay with the convergence of $\omega^n$ to $\overline{\omega}$, it follows that for every $\epsilon>0$ there exists $n$ such that, with probability at least $1-\epsilon$, the solution $\omega^n$ remains below the cutoff threshold $R$ for all $t \in [0,T]$. 
Hence, with high probability, the cutoff is never activated in \eqref{main_ito_cutoff_scaling}, and $\omega_T^n$ is small in the critical space $B^{-1/2}_{g,2,4}$. Then a ``small data imply global well-posedness''-type result for \eqref{main_ito}, applied at time $T$ (see \autoref{prop:small_implies_global} with $p=4$), implies the claim of \autoref{thm:main}.
We emphasize that the convergence in \eqref{eq:convergence_prob_intro} is highly non-trivial, as there are explicit limitations for the spaces in which the scaling limits \eqref{main_ito_cutoff_scaling}-\eqref{main_det_cutoff_intro} can take place even in the linear case, see \cite[Remark 3.4]{agresti_anomalous_2024}.

\smallskip

Although successfully applied to many SPDEs (see e.g., \cite{agresti_global_2024,agresti_kniely_tang_2026,FGL21_delayed,lange2024regularization}, and the references therein), in the current situation the above program presents the following difficulties: 
\begin{enumerate}
\item {\rm (Singular stretching: Absence of uniform $L^2$-bounds).} From the construction of $(\sigma^n)_n$ in \autoref{sec_ito}, one can check that, as $n \to \infty$, the noise becomes increasingly singular, enhancing stretching effects and preventing standard $L^2$ estimates (see \cite[Appendix 2]{flandoli_high_2021}). More precisely, for all $\gamma>0$, it holds that 
\begin{align*}
\sup_{n\geq 1}\|(\sigma^n_k)_{k\in I}\|_{L^\infty(\T^3;\ell^2)}<\infty \quad \text{ and }\quad \liminf_{n\to \infty}\|(\sigma^n_k)_{k\in I}\|_{C^\gamma(\T^3;\ell^2)}=\infty;
\end{align*}
cf., \cite[Proposition 4.1]{agresti2025role}. 
Since $\LL_{\sigma_k^n} \omega^n =-\nabla \times (\sigma_k^n \times \omega^n)$, one can expect uniform-in-$n$ energy estimates only in the space $H^{-1}$ (at least on bounded sets $[0,T]$ with $T<\infty$, see item \eqref{it:lack_energy} below).
Indeed, in this topology, the balance between noise and dissipation is restored, and the noise is of lower order than the hyperdissipative operator $A$ in $H^{-1}$:
\begin{align}
\label{eq:smallness_condition_intro}
    \sum_{k}\| \LL_{\sigma_k^n}\omega \|_{H^{-1}}^2
\leq \frac{1}{4} \langle -A \omega,\omega\rangle_{H^{-1}}
+C\|\omega \|_{H^{-1}}^2
\end{align}
for some constant $C \geq 0$ independent of $n$. This inequality shows that the noise acts as a perturbation of the dissipative dynamics.

\item {\rm (Energy supercriticality)}
Note that the $H^{-1}$-estimate is too weak to obtain the convergence \eqref{eq:convergence_prob_intro}. Indeed, an a-priori estimate in $H^{-1}$ together with a compactness argument, only leads to \eqref{eq:convergence_prob_intro} with $H^{r}_g$ replaced by $H^{-1-\delta}$ for all $\delta>0$ (cf., \cite{FGL21_delayed,flandoli_high_2021}). To obtain convergence in such a strong topology, exploiting the relative smallness of the noise compared to $A$ (see \eqref{eq:smallness_condition_intro}), we refine the arguments in \cite{agresti_global_2024}, where stochastic maximal $L^p$-regularity in time (for $p \gg 2$) is used to obtain uniform-in-$n$ a priori bounds for $\omega^n$ in a Besov space compactly embedded in $H^r_g$.

\item\label{it:lack_energy} {\rm (Lack of energy dissipation).} Due to the special structure of the noise in the SPDE \eqref{main_ito_cutoff_scaling}, the usual pathwise
energy dissipation inequality of the energy $\|\omega_t\|_{H^{-1}}\eqsim \|u_t\|_{L^2}$ is no longer available. This is central to applying classical methods to deduce global well-posedness from the smallness of the initial data; see e.g., \cite[Section 4]{agresti_global_2024}, and condition (b) in \cite[Theorem 1.5]{FGL21_delayed}.
To overcome this difficulty, we exploit further the hyperviscous structure of the operator $A$ in \eqref{main_ito_cutoff_scaling}. More precisely, in \autoref{lem:high_mode_energy_defect}, we will prove that, for all $\Lambda<n/2$ large (e.g., $\Lambda=\sqrt{n}$),  
\begin{equation}
\label{eq:almost_cancellation_energy}
 \Big|\sum_{k\in I}\|\LL_{\sigma_k^n}\omega\|_{H^{-1}}^2-\|\omega\|_{L^2}^2\Big|
 \lesssim \Big(\frac{\Lambda}{n}+\frac1{g(\Lambda)}\Big)\|\omega\|_{H_g^0}^2,
\end{equation}
which shows that the difference between the contributions of the It\^o-Stratonovich corrector and It\^o's correction to the energy balance of $\omega^n$ is small when measured in the domain $H^0_g$ of the operator $A$.
This will allow us to obtain global well-posedness for small data for the SPDE  in \autoref{prop:small_implies_global} with constants that are uniform in $n$. 
In light of \autoref{r:lack_energy}, we do not believe the latter result is true in the absence of hyperviscosity (i.e., $g\equiv 1$).
\end{enumerate}

\subsubsection{Enhanced dissipation  by transport-stretching noise  (\autoref{thm_global})}
The appearance of the additional dissipation $\nu\Delta$ in the limiting PDE \eqref{main_det_cutoff_intro}, in the relatively weak norm discussed above, is the key phenomenon behind \autoref{thm:main}. 
Enhanced dissipation results as in \autoref{thm_global} provide a quantitative way to capture this effect. 
Let us emphasize that the ``almost cancellation'' in \eqref{eq:almost_cancellation_energy}, while sufficient to obtain global solutions for small initial data for \eqref{main_ito} uniformly in $ n$, does not directly yield an arbitrarily large decay rate for $\|\omega_t\|_{H^{-1}}$.
The proof of the latter combines the above strategy with a modified energy estimate that exploits the Fourier localization of the noise. Our proof differs from the previously used in the context of transport noise (see e.g., \cite{flandoli_quantitative_2024,luo_enhanced_2021}), since they are based on pathwise energy dissipation and the latter is not available for transport-stretching noise, see \eqref{it:lack_energy} in the list in \autoref{sss:strategy_main}. To recover the enhanced dissipation effect, we exploit the bounds underlying the global well-posedness of \autoref{thm:main}, together with an estimate for a modified norm that captures the enhanced dissipation effect of the transport-stretching noise by assigning different weights to low and high frequencies.

\smallskip

More precisely, for $p>4$, we consider the following truncation
\begin{align*}
\E_{\omega}(t)\coloneq
 \sup_{s\in[0,t]}\|\omega_s\|_{B^{-2/p}_{g,2,p}}^p
 +\int_0^t\left(\|\omega_s\|_{H_g^0}^p+\|\omega_s\|_{H_g^0}^4\right)\,\dd s,
\end{align*}
and we replace the cutoff in \eqref{main_ito_cutoff_scaling} by
$\psi_R(t,\omega^n)=\psi(\E_{\omega^n}(t)/R)$, where $\psi$ is smooth and satisfies $\psi|_{[0,1]}=1$ and $\psi|_{[2,\infty)}=0$; see \eqref{eq:ED_auxiliary_dynamics}--\eqref{eq:definition_modified_cutoff_psi}. 
The reason behind the $L^4(H^0_g)$-norm in $\E_{\omega^n}$ is to bound the nonlinear term $\LL_{u^n}\omega^n$ in the energy estimate; see \eqref{eq:bound_Yt_intro} below. A refinement of the estimates on the half-line used in the proof of \autoref{thm:main} shows that the solution with such modified truncation is a solution to the original problem with overwhelming probability for $R$ and $n$ sufficiently large. In the following, $\omega_t^n$ stands for the solution to  \eqref{main_ito_cutoff_scaling} with such cutoff. 

\smallskip

For $\Lambda<n/2$ large, let $P_{\leq\Lambda}$ (resp.\ $P_{>\Lambda}$) denote the Fourier projections onto lower (resp.\ strictly higher) frequencies than $\Lambda$, respectively. For a suitably small $\alpha\in(0,1)$, consider the modified energy at frequency $\Lambda$:
\begin{align*}
\widehat{\E}_\Lambda(\omega)
 =\|P_{\leq\Lambda}\omega\|_{H^{-1}}^2
  +\alpha\|P_{>\Lambda}\omega\|_{H^{-1}}^2.
\end{align*}
In particular, $\alpha\|\omega\|_{H^{-1}}^2\leq
\widehat{\E}_\Lambda(\omega)\leq\|\omega\|_{H^{-1}}^2$.
This energy captures the transfer of energy across the frequency $\Lambda$. 
By \autoref{hyp_noise}, the nonzero coefficients $\sigma_k^n$ act as a shift in Fourier frequencies by a vector of length at least $n$. In particular,
\begin{align*}
 P_{\leq\Lambda}\LL_{\sigma_k^n}P_{\leq\Lambda}=0.
\end{align*}
Therefore, while applying It\^o's formula to $\widehat{\E}_\Lambda(\omega_t^n)$, the contribution of low frequencies to It\^o's correction is multiplied by $\alpha$, while the corresponding dissipation due to $\nu\Delta$ is not affected; the contribution of high frequencies of the transport-stretching noise instead can be absorbed by the hyperviscous dissipation $A= \Delta g(\sqrt{-\Delta})$, since $g(\Lambda)\to\infty$. 
With this argument, in the proof of \autoref{prop:ED_truncated}, we can prove the following key estimate:

For given $\chi,b\geq1$ and an integer $m>2\vee b$, one can choose $\nu$ sufficiently large, $\alpha$
sufficiently small depending on $m$, and $\Lambda$ sufficiently large so that
\begin{equation}
\label{eq:bound_Yt_intro}
 \dd \widehat{\E}_\Lambda(\omega_t^n)^{m/2}+m\chi\widehat{\E}_\Lambda(\omega_t^n)^{m/2}\,\dd t \leq C_{b,g}\psi_R(t,\omega^n)\|\omega_t^n\|_{H_g^0}^4 \widehat{\E}_\Lambda(\omega_t^n)^{m/2}\,\dd t+\dd\M_t,
\end{equation}
where $\M$ is a continuous local martingale. Note that $\int_0^\infty\psi_R(s,\omega^n)\|\omega_s^n\|_{H_g^0}^4\,\dd s\leq 2R$ by construction. Therefore, Doob's maximal inequality and $b<m$ readily yield the exponential decay  
\begin{align*}
 \|\omega^n_t\|_{H^{-1}}
 \leq D e^{-\chi t}\|\omega_0\|_{H^{-1}}
 \quad\text{for all }t\geq0,
 \quad \text{ where }\quad 
 \EE [D^b]\lesssim_{R,b,g}1,
\end{align*}
for $n$ sufficiently large; see the arguments following \eqref{eq:ED_supermartingale} for details.

\subsection{Structure of the paper} The paper is organized as follows. In \autoref{sec:setup}, we fix some notation employed throughout the paper, as well as rigorously provide the concept of solutions to \eqref{main_strat} and state our main results. In \autoref{sec:well_posed}, we provide some local well-posedness results for \eqref{main_strat} as well as global ones and a priori bounds for its cutoff version introduced above. The convergence of \eqref{main_ito_cutoff_scaling} to \eqref{main_det_cutoff_intro} occupies \autoref{sec:scaling_limit}, while the proofs of \autoref{informal_thm1} and \autoref{informal_thm2} occupy \autoref{sec:proof_of_main} and \autoref{sec:proof_global} respectively. Finally, in \autoref{app_technical_res}, we recall for the convenience of the readers some results from \cite{agresti_nonlinear_2025}, as well as prove some interpolation properties related to our functional analytic framework and the \emph{almost cancellation} given by \eqref{eq:almost_cancellation_energy}.

\subsubsection*{AI Statement}
ChatGPT Astra suggested \autoref{lem:high_mode_energy_defect}, which allowed us to extend the results that already appeared in an arXiv preprint in June 2026. The latter was written without any use of LLM tools. All AI-suggested further revisions have been carefully reviewed by the authors before implementation. We take full responsibility for all the contents of the manuscript.

\subsubsection*{Funding} A.A.\ and E.L.\ are members of GNAMPA (INdAM). A.A.\ acknowledges support from INdAM through the GNAMPA 2026 project ``Fluidodinamica stocastica: irregolarità, trasporto e fenomeni di regolarizzazione''.

\section{Preliminaries}\label{sec:setup}
\subsection{Notation}
\label{ss:notation}
We begin by introducing some classical notation and recalling basic properties of operators in the periodic setting that will be used throughout the paper. For a comprehensive treatment of the material summarized in this subsection, we refer to the monographs \cite{temam1995navier,trie1995fun}.

\smallskip

Recall that $\mathbb{T}^3=\mathbb{R}^3/\mathbb{Z}^3$ denotes the three-dimensional torus, and let $\mathbb{Z}_0^3=\mathbb{Z}^3\setminus\{0\}$ be the integer lattice without the origin. We introduce a partition
\begin{align*}
\mathbb{Z}_0^3=\Lambda_{+}\cup\Lambda_{-},
\end{align*}
such that $\Lambda_{+}\cap\Lambda_{-}=\emptyset$ and $\Lambda_{+}=-\Lambda_{-}$. We also set
\begin{align*}
I\coloneq2\pi\mathbb{Z}_0^3\times\{1,2\}.
\end{align*}
For every $m\in2\pi\mathbb{Z}_0^3$ we fix unit vectors $a_{m,j}$, $j\in\{1,2\}$, such that 
\begin{align*}
\left\{m/|m|,a_{m,1},a_{m,2}\right\}
\end{align*}
forms an orthonormal basis of $\mathbb{R}^3$. We additionally require the symmetry property
\begin{align*}
a_{-m,j}=a_{m,j}, \qquad j\in\{1,2\}.
\end{align*}
For $k=(m,j)\in I$ we define $-k\coloneq(-m,j)$ and set
\begin{align*}
e_k(x)=a_{m,j}e^{im\cdot x}.
\end{align*}
If $a,b>0$, we write $a\lesssim b$ if there exists a constant $C>0$ such that $a\leq Cb$. We write $a\lesssim_\xi b$ when we wish to emphasize the dependence of the constant $C$ on a parameter $\xi$.

\smallskip

Let $\left(H^{s,p}(\mathbb{T}^3),\|\cdot\|_{H^{s,p}}\right)$, $s\in\mathbb{R}$ and $p\in(1,\infty)$, denote the Bessel potential spaces of periodic functions with zero mean. In the case $p=2$ we simply write $H^s(\mathbb{T}^3)$ instead of $H^{s,2}(\mathbb{T}^3)$ and denote by 
\begin{align*}
\langle f,g\rangle_{H^s}=\widehat{f}(0)\overline{\widehat{g}(0)}+\textstyle\sum_{k\in 2\pi\Z^3_0} |k|^{2s}\widehat{f}(k)\overline{\widehat{g}(k)}
\end{align*} 
the corresponding inner product, where $\widehat{f}(k)= \int_{\T^3}f(x) e^{-ik\cdot x}\dd x$.
In the case $s=0$, we write $L^p(\mathbb{T}^3)$ in place of $H^{0,p}(\mathbb{T}^3)$, and when $p=2$ we omit the subscript in the notation for both the norm and the inner product. With a slight abuse of notation, for $s>0$ we denote the duality pairing between $H^{-s}$ and $H^s$ by $\langle\cdot,\cdot\rangle$.

We similarly introduce the Bessel spaces of vector fields.
\begin{align*}
H^{s,p}(\mathbb{T}^3;\mathbb{R}^3)
=
\{u=(u^1,u^2,u^3)^t:\ u^1,u^2,u^3\in H^{s,p}(\mathbb{T}^3)\}.
\end{align*}
Again, when $s=0$ we write $L^p(\mathbb{T}^3;\mathbb{R}^3)$ instead of $H^{0,p}(\mathbb{T}^3;\mathbb{R}^3)$ and omit subscripts in the notation of the norm and scalar product when $p=2$.

We denote by $H^{s,p}_{\div}$ the closed subspace of $H^{s,p}(\mathbb{T}^3;\mathbb{R}^3)$ consisting of zero-mean divergence-free vector fields, endowed with the norm induced by $H^{s,p}(\mathbb{T}^3;\mathbb{R}^3)$. In the case $p=2$ (resp. $s=0$) we simply write $H^s_{\div}$ (resp. $L^p_{\div}$).
It is well known that the Stokes operator
\begin{align*}
\Delta:H^2_{\div}\subset L^2_{\div}\rightarrow L^2_{\div}
\end{align*}
is linear, closed, self-adjoint, and of negative type. Moreover, the family $(e_k)_{k\in I}$ forms an orthonormal system in $L^2_{\div}$ consisting of eigenfunctions of $-\Delta$.

We also denote by
\begin{align*}
\Pi \coloneq \mathrm{Id}- \nabla \Delta^{-1}\nabla\cdot :H^{s,p}(\mathbb{T}^3;\mathbb{R}^3)\rightarrow H^{s,p}_{\div}
\end{align*}
the Leray projection and by
\begin{align*}
K\coloneq \nabla\times (-\Delta)^{-1}:H^{s,p}_{\div}\rightarrow H^{s+1,p}_{\div}
\end{align*}
the Biot-Savart operator, without distinguishing the values of $s$ and $p$ in the notation.

\smallskip

We conclude this subsection by introducing some standard notation for stochastic processes taking values in separable Banach spaces. Let $(\Omega,\mathcal{F},(\mathcal{F}_t)_{t\geq0},\mathbb{P})$ be a filtered probability space and let $Z$ be a separable Banach space with norm $\|\cdot\|_Z$. Throughout the paper, we assume, unless otherwise specified, that $(\Omega,\mathcal{F},\mathbb{P})$ is complete and that $(\mathcal{F}_t)_{t\geq0}$ is a right-continuous filtration such that $\mathcal{F}_0$ contains all $\mathbb{P}$-null sets.

Given $T_0,T\in [0,+\infty],\ T>T_0$, we denote by $L_{\mathcal{F}}^{0}(T_0,T;Z)$ the space of progressively measurable processes with values in $Z$, and by $L_{\mathcal{F}}^{p}(T_0,T;Z)$, $p\in[1,\infty)$, the space of progressively measurable processes $(X_t)_{t\in[T_0,T]}$ with values in $Z$ such that
\begin{align*}
\mathbb{E}\Big[\int_{T_0}^T\|X_t\|_Z^p\,\mathrm{d}t\Big]<\infty.
\end{align*}
Finally, for $t\geq0$ and $p\in (0,\infty)$, $L^p_{\mathcal{F}_t}(\Omega;Z)$ denotes the space of $\mathcal{F}_t$-measurable random variables with values in $Z$ with finite $p$-th moment, and with the obvious modifications in case $p\in \{0,+\infty\}$.
\subsection{Functional Analytic Setup}
In this section, we introduce the class of hyperdissipative operators $A$ considered in this work, of which \eqref{logaritmicA} provides a representative example, together with the functional framework naturally associated with them. We begin by stating a set of assumptions.
\begin{hypothesis}\label{hp_correction}
    Let $g:\R^+\rightarrow \R^+$ be a function satisfying
    \begin{enumerate}
        \item\label{hp_correction1} $g$ is increasing and $g(1)>0.$
        \item\label{hp_correction2} $\sup_{r>0}\frac{g(2r)}{g(r)}<+\infty.$
    \item\label{hp_correction3} $\lim_{r\rightarrow +\infty}g(r)=+\infty,\ \lim_{r\rightarrow +\infty}\frac{g(r)}{r^2}=0.$
    \end{enumerate}
\end{hypothesis}

\begin{remark}
    Typical examples satisfying the above assumptions are 
    \begin{equation*}
        g(r)=\log(1+r),\quad g(r)=r^{\gamma},\ 0< \gamma<2.
    \end{equation*}
\end{remark}
For every $g$ as in \autoref{hp_correction}, we define a hyperdissipative operator $A$ as follows.

\begin{definition}\label{Defintition_operator}
We denote by $A$ the diagonal operator with Fourier symbol $-\lvert m\rvert^2 g(|m|)$, namely 
\begin{align*}
    Ae^{im\cdot x}=-\lvert m\rvert^2 g(|m|) e^{im\cdot x},\qquad m\in2\pi \mathbb{Z}^{3}.
\end{align*}
\end{definition}
Throughout the paper, \autoref{hp_correction} will always be assumed, even when not explicitly stated. We underline that conditions \eqref{hp_correction1} and \eqref{hp_correction3} imply that $A$ is stronger than $\Delta$ at high frequencies, while \eqref{hp_correction2} is a technical assumption only needed to characterize interpolation spaces. Before stating the properties of the operator $A$, we need to briefly recall some notation from \cite[Section 3]{gubinelli2015lectures}. Let $(\chi,\rho)$ be a dyadic partition of unity. Define
\begin{align*}
    \rho_{-1}\coloneq \chi, \qquad 
\rho_j(\xi)\coloneq \rho(2^{-j}\xi), \qquad j\geq0.
\end{align*}
For $u\in\mathscr{S}'(\T^3)$ we define the Littlewood-Paley blocks
\begin{align*}
\Delta_j u
=
\mathcal{F}^{-1}(\rho_j \mathcal{F}u)
=
\sum_{m\in2\pi\Z^3}
e^{im\cdot x}\widehat{u}(m)
\rho_j(m).
\end{align*}
With this notation at hand, we can introduce the main function spaces used in this work, which are tailored to the operator $A$ defined in \autoref{Defintition_operator} (see \autoref{semigroup} and \autoref{generation_semigroup} below).

\begin{definition}\label{def_function_spaces}
    For $s\in \R$ we denote by 
    \begin{align*}
        H^{s}_g=\Big\{h\in \mathscr{S}'(\T^3;\R^3):\operatorname{div}h=0,\int_{\T^3}h \,\mathrm{d}x=0, \sum_{k=(m,j)\in I} \lvert m\rvert^{2s}g(|m|)^{s+1}|\langle h,e_{k}\rangle|^2<+\infty\Big\}
    \end{align*}
    endowed with its natural norm. 
    For $s\in \R,\ p\in [1,+\infty)$ we set
    \begin{align*}
        B^{s}_{g,2,p}=\Big\{h\in \mathscr{S}'(\T^3;\R^3)\,&\left.:\, \operatorname{div}h=0,\int_{\T^3}h\, \mathrm{d}x=0,\right.\\
        &\ \ \sum_{j\geq -1} \left(2^{s j}g(2^j)^{\frac{s+1}{2}}\norm{\Delta_j h}_{L^2(\T^3;\R^3)}\right)^{p}<+\infty \Big\}
    \end{align*}
    endowed with its natural norm, with the obvious modification in the case $p=\infty$.
\end{definition}
\begin{remark}
    Clearly $B^{s}_{g,2,p_1}\hookrightarrow B^{s}_{g,2,p_2}$ whenever $p_1\leq p_2$.
\end{remark}

We further introduce the spaces
\begin{equation*}
    X_0=H^{-2}_g,
\qquad
X_{1/2}=H^{-1}_g\simeq H^{-1}_{\div},
\qquad
X_1=H^0_g .
\end{equation*}
The spaces defined above and the operator $A$ satisfy the following properties. The proofs follow from standard arguments and are postponed to \autoref{sec:interpolation}.
\begin{lemma}\label{lemma_interpolation}
For every $s\in \R$ it holds
\begin{align}\label{identification_p2}
H^s_g\simeq B^{s}_{g,2,2}. 
\end{align}
Given $p_0,p_1,p\in [1,+\infty],\ s_0,s_1\in \R: s_0< s_1$ and $\theta\in (0,1)$ it holds
\begin{align}\label{real_interpolation}
    (B^{s_0}_{g,2,p_0}, B^{s_1}_{g,2,p_1})_{\theta,p}=B^{\theta s_1+(1-\theta)s_0}_{g,2,p}
\end{align}
\end{lemma}

\begin{lemma}\label{generation_semigroup}
The operator
$
    A:\Do(A)=X_1\subseteq X_0\rightarrow X_0
$ is linear and closed. Moreover, $-A$ is self-adjoint and satisfies
\begin{align}\label{inequality_coercivity}
    \langle -A x,x\rangle_{X_0}\geq \norm{x}_{X_{1/2}}^2\geq g(1)\norm{ x}_{X_0}^2\quad \text{ for all }x\in X_1.
\end{align}  
\end{lemma}

As a consequence of \autoref{lemma_interpolation} and \autoref{generation_semigroup}, together with standard semigroup and interpolation results (see in particular \cite[Chapter 2]{Lunardi_book} and \cite{weis2006h}), we obtain the following.

\begin{corollary}\label{semigroup}
$A:X_1\subseteq X_0\rightarrow X_0$ generates an analytic semigroup of negative type on $X_0$ which we denote by $e^{At}$. Moreover, $-A$ admits a bounded $H^{\infty}$ calculus with angle zero. Finally, for $\theta\in (0,1),$ and  $p\in [1,+\infty]$, setting 
\begin{align*}
    \Do_A(\theta,p)=\Big\{x\in X_0:t\mapsto \Big\|t^{1-\theta-1/p}A e^{At}x\Big\|_{X_0}\in L^p(0,1)\Big\}
\end{align*}
endowed with the norm $\norm{x}_{\Do_A(\theta,p)}=\norm{x}_{X_0}+\|t^{1-\theta-1/p}A e^{At}x\|_{L^p(0,1;X_0)}$,
it holds
\begin{align}\label{fractional_powers}
    [X_0,X_1]_{\theta}&=\Do((-A)^{\theta})= H^{2(\theta-1)}_g,\\
    \label{trace_spaces}
    \left(X_0,X_1\right)_{\theta,p}&=\Do_A(\theta,p)=B^{2(\theta-1)}_{g,2,p},\\
    \label{embedding}
    \Do_A(\theta,1)&\hookrightarrow \Do((-A)^{\theta})\hookrightarrow \Do_A(\theta,\infty).
\end{align}
\end{corollary}
\subsection{Main Results}\label{sec_ito}
Before stating our main results, we need to rigorously introduce system \eqref{main_strat} and define the noise coefficients. As already mentioned in \autoref{sec:intro}, we employ an isotropic, divergence-free noise. We choose the coefficients $(\sigma_k)_{k\in I}$ having the following structure
\begin{align*}
    \sigma_k(x)= \sigma_{m,j}(x)=\theta_m a_{m,j}e^{im\cdot x},\qquad \theta_m=\theta_{-m}.
\end{align*}
Next, we take $(W_t^k)_{k\in I}$ a sequence of $\C$-valued Brownian motions, such that 
\begin{align*}
    [W^{(m,i)}_t, \overline{W^{(m', j)}_t}]= 2t \delta_{m,m'}\delta_{i,j} , \qquad \overline{W^{(m, j)}_t}= W^{(-m, j)}_t.
\end{align*}
In this way, our noise is real, indeed 
\begin{equation*}
    \overline{\sigma_k} = \sigma_{-k}, \qquad \overline{W}_t\coloneq \overline{\sum_{k\in I}\sigma_kW_t^k} = \sum_{k\in I}\overline{\sigma_k} W_t^{-k} = \sum_{k\in I} \sigma_{-k}W_t^{-k} = W_t.
\end{equation*}
Finally, we make an explicit choice for the coefficients to simplify computations:

\begin{hypothesis}\label{hyp_noise}
There exists $n\in \N,\ n\geq 4$ for which $\theta=(\theta_m)_{m\in 2\pi\Z^3_0}$ is given by
\begin{equation}\label{def_coefficients}
 \theta_m \coloneq \frac{1}{Z_n}\one_{[n\leq |m| \leq 2n]} |m|^{-3/2}\quad ,
\end{equation}
where $Z_n= \big(\frac{2}{3}\sum_{m\in 2\pi\Z_0^3}\one_{[n\leq |m| \leq 2n]} |m|^{-3}\big)^{1/2}$ denotes the $\ell^2$-normalization.
\end{hypothesis}

The above assumption is only made for exposition convenience. The reader can readily check that other choices are possible; see e.g., \cite[Remark 5.7]{flandoli_high_2021}. 
As is customary, to interpret \eqref{main_strat} we first rewrite it in It\^o form.
Due to \autoref{hyp_noise}, by \cite[Lemma 2.13]{butori_meanfield_2025}, equation \eqref{main_strat} is formally equivalent to the following stochastic partial differential equation: 
\begin{equation}\label{main_ito}
\begin{cases}
    \mathrm{d}\omega + \left( \LL_u \omega - A\omega - \nu\Delta \omega \right)\mathrm{d}t + \sqrt{\nu}\sum_{k\in I}\LL_{\sigma_k} \omega\, \mathrm{d}W_t^k =0, \\
    u= K[ \omega]. 
    \end{cases}
\end{equation}
Our analysis will focus on the properties of the above SPDE. We first introduce a proper notion of solution for \eqref{main_ito}.

\begin{definition}\label{def:local_well_posed}
    Given $\omega_0 \in L^0_{\mathcal{F}_0}(\Omega, H^{-1}_{\operatorname{div}})$, $\theta$ satisfying \autoref{hyp_noise}, $p\in [2, +\infty)$, a stopping time $\tau:\Omega\rightarrow [0, +\infty]$ and a progressively measurable process $\omega: \Omega\times [0, \tau) \rightarrow H^{-2}_g$.  
    \begin{itemize}
        \item[$\bullet$] $(\omega,\tau)$ is said to be a \emph{local $p$-solution} of equation \eqref{main_ito} if the following are satisfied:
    \begin{itemize}
        \item[-] $\omega\in L^p_{\mathrm{loc}}([0, \tau); H^{0}_g) $ almost surely;
        \item[-] $\LL_u\omega \in L^p_{\mathrm{loc}}([0, \tau); H^{-2}_g) $ almost surely;
        \item[-] almost surely, it holds for all $t\in [0,\tau)$:
        \begin{equation*}
            \omega_t- \omega_0 + \int_0^t \left(\LL_{u_s}  -A - \nu \Delta\right)\omega_s \,\mathrm{d}s + \sqrt{\nu}\sum_{k\in I} \int_0^t\LL_{\sigma_k}\omega_s \,\mathrm{d}W_s^k=0,
        \end{equation*}
        where the equality is intended as elements of $H^{-2}_g$.
    \end{itemize}
    \item[$\bullet$] A local $p$-solution $(\omega,\tau)$ to \eqref{main_ito} is said to be a \emph{unique local $p$-solution} to \eqref{main_ito} if for any local solution $(\omega',\tau')$ we have $\omega=\omega'$ a.e. on $[0,\tau\wedge\tau')\times \Omega$.
    \item[$\bullet$] A unique local $p$-solution $(\omega,\tau)$ to \eqref{main_ito} is said to be a \emph{unique maximal $p$-solution} to \eqref{main_ito} if for any local solution $(\omega',\tau')$ we have $\tau'\leq \tau$ almost surely and $\omega=\omega'$ a.e. on $[0,\tau\wedge\tau')\times \Omega$.
    \end{itemize}
\end{definition}
\begin{remark}
Considering the cutoff version of \eqref{main_ito}, namely
\begin{equation}\label{main_ito_cutoff}
\begin{cases}
    \mathrm{d}\omega + \big[ \phi_{R}(\|\omega\|_{H^r_g})\LL_u \omega -A\omega - \nu\Delta \omega \big]\mathrm{d}t + \sqrt{\nu}\sum_{k\in I}\LL_{\sigma_k} \omega\, \mathrm{d}W_t^k =0 \\
    u= K[\omega],
    \end{cases}
\end{equation}
where $R>0,\ -\frac{1}{2}<r<0$ and $\phi_{R}:[0,+\infty)\rightarrow [0,1]$ is a smooth, non-increasing function such that
\begin{align*}
    \phi_{R}(\rho)=\begin{cases}
        1\quad \text{if } \rho\leq R\\
        0\quad \text{if } \rho> R+1,
    \end{cases}
\end{align*}
an analogous definition holds, just replacing the nonlinearity $\LL_u\omega$ with $\phi_R(\|\omega\|)\LL_u \omega.$
\end{remark}
Existence and uniqueness of maximal local $p$-solutions is proved in \autoref{sec:well_posed} for $p\geq 4$ and $\omega_0 \in B^{{-2/p}}_{g,2,p}$. In case of $p>4$ we can run the program described in the introduction by \autoref{informal_thm1} and \autoref{informal_thm2} leading to the following results: 

\begin{theorem}[Global smooth solutions of LNSEs by transport-stretching noise]
\label{thm:main}
Let $p\in ( 4, \infty)$, $\eps\in(0,1)$, $M\geq 1$. Then there exists $\nu >0$, and $(\theta_k)_{k\in I}$ satisfying \autoref{hyp_noise} and independent of time such that for all initial data 
\begin{align}
\label{eq:assumption_initial_data_theorem}
    \omega_0 \in B^{{-2/p}}_{g,2,p} \qquad \text{ satisfying }\qquad \|\omega_0\|_{B^{{-2/p}}_{g,2,p}} \leq M,
\end{align}
there exists a unique maximal $p$-solution $(\omega, \tau)$ of equation \eqref{main_ito} such that 
\begin{align}
\label{eq:lifetime_infinite_statement}
    \PP(\tau =\infty) \geq 1-\eps.
\end{align}
Moreover for all $\theta, \theta_1\in [0,\frac{1}{2})$ and $ \theta_2\in \N$, it holds that, a.s.,
\begin{align}
\label{eq:smoothness1_thm}
\omega &\in H^{\theta, p}_{{\rm loc}}([0, \tau); H^{-2\theta}_g)\cap C([0, \tau); B^{{-2/p}}_{g,2,p}) \\ 
\label{eq:smoothness2_thm}
\omega &\in C^{\theta_1, \theta_2}_{{\rm loc}} ((0, \tau)\times \T^3;\R^3).
\end{align} 
\end{theorem}

The key point in the above result is \eqref{eq:lifetime_infinite_statement}, which ensures the long-time existence of unique solutions to \eqref{main_ito}, while \eqref{eq:smoothness1_thm}-\eqref{eq:smoothness2_thm} concern their smoothness and follow from the recent theory of SPDE in critical spaces; see \autoref{sec:av_theory} for details. In particular, \eqref{eq:smoothness2_thm} ensures that the maximal $p$-solution $(\omega, \tau)$ instantaneously becomes classical in space.

\smallskip

The following result concerns the enhanced dissipation effect of the transport-stretching noise on the velocity field $u$ associated with the above maximal $p$-solution that is global in time with large probability.

\begin{theorem}[Enhanced dissipation of velocity for LNSEs by transport-stretching noise]
\label{thm_global}
Let $p\in ( 4, \infty)$, $\eps\in(0,1)$ and $M,\chi,b\geq 1$. Then, 
there exists $\nu>0$ and noise coefficients $\theta$ satisfying \autoref{hyp_noise} such that, for all initial data as in \eqref{eq:assumption_initial_data_theorem} the maximal $p$-solution $(\omega, \tau)$ of equation \eqref{main_ito} as in \autoref{thm:main} satisfies
\begin{equation}
\|\omega_t\|_{H^{-1}}\leq D e^{-\chi t}
\|\omega_0\|_{H^{-1}} \ \text{ for all }t<\mu,
\end{equation} 
where $\mu\leq \tau$ is a stopping time, and $D$ a random constant satisfying:  
\begin{align*}
    \PP(\mu =\infty) \geq 1-\eps \qquad \text{ and }\qquad \EE [D^b] \lesssim_{p,\eps,M,\chi,b} 1.
\end{align*}
\end{theorem}

As commented around \eqref{eq:velocity_vorticity_intro}, $\|u\|_{L^2}\eqsim \|\omega\|_{H^{-1}}$, and therefore the above quantifies the enhanced dissipation effect of the transport-stretching noise on the Navier-Stokes dynamics. 

\section{Local Well-Posedness Results}
\label{sec:well_posed}
We collect in this section the basic well-posedness results on \eqref{main_ito} and \eqref{main_ito_cutoff}; we refer to \autoref{sec:proof_local_well} below for their proofs.

As a direct consequence of \autoref{SMR_maximal_thm}, we obtain the local well-posedness of \eqref{main_ito}.

\begin{proposition}\label{prop:local_well_nonlinear}
    Let $4\leq p<+\infty$ and $\theta$ satisfying \autoref{hyp_noise}. Then for all $\omega_s\in L^0_{\mathcal{F}_s}\left(\Omega, B^{-2/p}_{g,2,p}\right),$ there exists a unique maximal $p$-solution $(\omega,\tau)$ of \eqref{main_ito} satisfying $a.s.$, $\tau>s$ and
    \begin{align}
        \omega&\in H^{\theta,p}_{{\rm loc}}([s,\tau);H_g^{-2\theta})\cap C([s,\tau);B^{-2/p}_{g,2,p})\quad \text{for each}\ \theta\in [0,1/2). 
    \end{align}
    Moreover, the solution $(\omega,\tau)$ instantaneously regularizes in
time and space: a.s.\
\begin{align}
\label{eq:instantaneous_regularization_omega}
    \omega\in C^{\theta_0,\theta_1}_{{\rm loc}}((s,\tau)\times \T^3;\R^3) \ \ \text{ for all }\theta_0\in [0,1/2),\ \theta_1<+\infty.
\end{align}
\end{proposition}

In the case of the truncated equation \eqref{main_ito_cutoff}, global well-posedness holds.
\begin{proposition}\label{prop:global_well_cutoff}
 Let $4< p<+\infty$ and $-\frac{1}{2}<r<-\frac{2}{p}$. For all $\omega_0\in B^{-2/p}_{g,2,p},$ there exists a unique maximal $p$-solution $(\omega,\tau)$ of \eqref{main_ito_cutoff} satisfying a.s., $\tau=+\infty$ and
 \begin{align*}
     \omega&\in H^{\theta,p}_{{\rm loc}}([0,+\infty);H_g^{-2\theta})\cap  C([0,+\infty);B^{-2/p}_{g,2,p}) \ \   \text{ a.s.\ for all } \theta\in [0,1/2).
 \end{align*}
 Moreover, for all $T>0$, 
    \begin{align}\label{a_priori_estimate_cutoff}
     \mathbb{E}\Big[\sup_{s\in [0,T]}\|\omega_s\|^p_{B^{-2/p}_{g,2,p}}\Big]+\mathbb{E}\Big[\int_0^T \|\omega_s\|_{H^0_g}^p \,\mathrm{d}s\Big]\leq C(T)\big(1+\|\omega_0\|_{B^{-2/p}_{g,2,p}}^p\big) ,
    \end{align}
    where $C$ is independent of frequency localization $n$ of the noise (cf. \autoref{hyp_noise}). 
\end{proposition}

\begin{remark}\label{rem:stoch2det}
    The analogous statement of \autoref{prop:global_well_cutoff}, without the expected value, holds as a consequence also for the deterministic system with cutoff, obtained by setting $\theta_m \equiv 0$. 
\end{remark}

\subsection{Proofs of well posedness for the nonlinear problems}\label{sec:proof_local_well}
In order to prove \autoref{prop:local_well_nonlinear} and  \autoref{prop:global_well_cutoff}, we use the abstract results of \cite{agresti_nonlinear_2022}, which we recall in \autoref{sec:av_theory}.
First, we need two lemmas to check \autoref{SMR_nonlinear_hp}.
\begin{lemma}\label{lem:nonlinear_est}
    For all $g$ verifying \autoref{hp_correction} and every $\xi, \omega \in H^{-1/4}_g$, denoting $v=K[ \xi], \ u = K[\omega]$ it holds 
    \begin{equation}
        \|\nabla \times (\div (u\otimes v))\|_{H^{-2}_g}\lesssim \|\omega \|_{H_g^{-1/4}}\|\xi \|_{H_g^{-1/4}}. 
    \end{equation}
    As a consequence, it holds 
    \begin{equation}
        \|\LL_u\omega - \LL_{v}\xi\|_{H^{-2}_g}\lesssim ( \|\omega\|_{H_g^{-1/4}} + \|\xi\|_{H_g^{-1/4}})\|\omega - \xi\|_{H^{-1/4}_g}.
    \end{equation}
\end{lemma}
\begin{proof}
    The second statement easily follows from the first one thanks to the bilinearity of the operator $\LL$. Indeed, it holds
    \begin{equation*}
        \LL_u\omega - \LL_{v}\xi =  \nabla \times (\div(u\otimes u- v\otimes v)) = \nabla \times \left(\div \left[(u-v)\otimes u + v \otimes (u-v)\right]\right).
    \end{equation*}
    We prove the first statement. We have 
    \begin{align*}
         \|\nabla \times (\div (u\otimes v))\|_{H^{-2}_g} \lesssim \|u\otimes v\|_{H^{0}_{g^{-1}}} .
    \end{align*}
    Since $g$ is increasing, we have the embedding $L^2_{\div} \hookrightarrow {H^{0}_{g^{-1}}}$. Moreover, we have the usual Sobolev embedding $H^{3/4}\hookrightarrow L^4$, therefore we obtain 
    \begin{align*}
        \|u\otimes v\|_{H^{0}_{g^{-1}}} \lesssim \|u\otimes v\|_{L^2} \leq  \|u\|_{L^4}\|v\|_{L^4} \lesssim \|u\|_{H^{3/4}} \|v\|_{H^{3/4}} \lesssim \|\omega\|_{H^{-1/4}}\|\xi\|_{H^{-1/4}}.
    \end{align*}
    Finally, since for all $s>-1$ it holds $H^s_g \hookrightarrow H^s$, we conclude  
    \begin{align*}
        \|\nabla \times (\div (u\otimes v))\|_{H^{-2}_g} \lesssim \|\omega\|_{H^{-1/4}_g} \|\xi\|_{H^{-1/4}_g}.
    \end{align*}
\end{proof}
\begin{remark}
    The computation of the previous lemma is almost sharp if $g(k)\sim \log(1+|k|)$; however, it is suboptimal if $g$ contains powers of $|k|$. Assume $g(|k|)= |k|^{2(\gamma-1)}\log(1+|k|)$ for some $\gamma \in (1,2)$; then we have, by repeated Sobolev embeddings,
        \begin{align*}
\|u\otimes v\|_{H^0_{g^{-1}}}
&\lesssim \|u\otimes v\|_{H^{1-\gamma}}\lesssim \|\omega\|_{H^{1/4-\gamma/2}}
           \|\xi\|_{H^{1/4-\gamma/2}}.
\end{align*}
    Finally, a straightforward inspection of the definition of the $H^s_g$ norm yields
    \begin{equation*}
       \|\omega\|_{H^{1/4-\gamma/2}}\lesssim \|\omega\|_{H_g^{\frac{5-6\gamma}{4\gamma}}},
    \end{equation*}
    and similarly for $\xi$. Note that ${\frac{5-6\gamma}{4\gamma}}<-\frac{1}{4}$ for $\gamma>1$. We will not stress this loss of sharpness in the sequel, since our aim is mainly to treat the case $\gamma =1$. 
\end{remark}
\begin{lemma}\label{lem:nonlinear_est_cutoff}
    For all $g$ verifying \autoref{hp_correction}, every $r \in (-1/2, 0)$  and every $\omega, \xi \in H^0_g$, denoting $u=K[\omega]$, $v=K[\xi]$, there exists $\delta_0 \in (r,0) $ and $k(r, \delta_0) \in (0,1)$ such that for every $\delta >\delta_0$
    \begin{equation}\label{ineq:nonlinear_cutoff_interp}
        \|\LL_u \omega \|_{H^{-2}_g} \lesssim \|\omega\|^{1-k}_{H^\delta_g}\|\omega \|^{1+k}_{H^r_g}.
    \end{equation}
    As a consequence, denoting $F_{R,r}(\omega)\coloneq \phi_R(\|\omega\|_{H_g^r})\LL_u \omega$, it holds
    \begin{equation}\label{ineq:lipschitz_cutoff}
        \|F_{R,r}( \omega) - F_{R,r}( \xi) \|_{H^{-2}_g} \lesssim (1 +\|\xi\|_{H^\delta_g})\|\omega - \xi\|_{H^r_g} + (\|\omega\|_{H^{-1/4}_g}+\|\xi\|_{H^{-1/4}_g})\|\omega - \xi\|_{H^{-1/4}_g}. 
    \end{equation}
\end{lemma}
\begin{proof}
    Again, the second statement is implied by the first one and by \autoref{lem:nonlinear_est}. Indeed, it holds
    \begin{align*}
        F_{R,r}( \omega) - F_{R,r}( \xi) = \left(\phi_R(\|\omega\|_{H_g^r}) - \phi_R(\|\xi\|_{H_g^r})\right)\LL_u \omega + \phi_R(\|\xi\|_{H_g^r})\left(\LL_u\omega - \LL_{v}\xi\right).
    \end{align*}
    Thus, we only need to control the first term. Without loss of generality, we can assume that $\|\omega\|_{H^r_g} \leq R+1$. Indeed, if this does not hold, either $\|\xi\|_{H^r_g} \leq R+1$ and then we can exchange the role of $\xi$ and $\omega$, or the term under study is zero, thanks to the properties of $\phi_R$.
     Thanks to the smoothness of $\phi_R$, we have 
    \begin{equation*}
        \left\|\left(\phi_R(\|\omega\|_{H_g^r}) - \phi_R(\|\xi\|_{H_g^r})\right)\LL_u \omega\right\|_{H^{-2}_g} \leq C\|\omega - \xi\|_{H_g^r}\|\LL_u \omega \|_{H_g^{-2}}.
    \end{equation*}
    Thus, \eqref{ineq:lipschitz_cutoff} follows from \eqref{ineq:nonlinear_cutoff_interp} and the Young inequality.
    To prove \eqref{ineq:nonlinear_cutoff_interp}, we just perform an interpolation. By \autoref{lem:nonlinear_est}
    \begin{equation*}
        \|\LL_u \omega \|_{H^{-2}_g} \lesssim \|\omega\|^2_{H^{-1/4}_g}. 
    \end{equation*}
    Thus, if $r\geq -1/4$, there is nothing to prove. If $r< -1/4$, we assume $\delta > -1/4$. Then interpolation gives
    \begin{align*}
        \|\omega\|^2_{H^{-1/4}_g} \lesssim  \|\omega\|^{2(1-\theta)}_{H^{r}_g} \|\omega\|^{2\theta}_{H^{\delta}_g}, \qquad \theta = \frac{1+4r}{4(r-\delta)}.
    \end{align*}
    Finally, since $r > -1/2$, it is possible to choose $-1/4 < \delta < 0$ so that $2\theta = 1-k$ and $2(1-\theta)= 1+k$ for some $k\in (0,1). $
\end{proof}

\subsubsection{Proof of \autoref{prop:local_well_nonlinear}}
Under \autoref{hyp_noise}, local well-posedness is a direct application of \autoref{SMR_maximal_thm} with linear operator given by $A + \nu \Delta$: the nonlinear hypothesis
is verified by \autoref{lem:nonlinear_est}, while stochastic maximal regularity follows from \autoref{SMR_linear_prop}.
Hence, \autoref{prop:local_well_nonlinear} follows by using the abstract results in \cite[Subsections 4.2 and 5.3]{agresti_nonlinear_2025} (see also \cite[Subsection 8.3]{agresti_nonlinear_2025} or \cite[Subsection 4.2]{MR4703457} for similar situations). To avoid repetition, we omit the details of the proof of \eqref{eq:instantaneous_regularization_omega}.

\subsubsection{Proof of \autoref{prop:global_well_cutoff}}
The local well-posedness follows as in the case without cut-off thanks to \autoref{lem:nonlinear_est_cutoff}. To show that the unique solution is global, we employ the blow-up criterion \autoref{SMR_blowup+instreg}: given $(\omega, \eta)$ the maximal solution, then for all $t< \infty$
\begin{align*}
\PP\Big(\eta <t, \ \sup_{r \in [0, \eta)}\|\omega_r\|_{B^{-\frac{2}{p}}_{g,2,p}}  + \int_0^\eta \|\omega_r\|_{H^0_g}^p \mathrm{d}r<+\infty\Big)=0.
\end{align*}
    Fix $t>0$ and choose localizing stopping times $\tau_n<\eta\wedge t$, with $\tau_n\uparrow\eta\wedge t$, such that the expectations on the left-hand side below
    are finite. Applying \autoref{SMR_linear_prop}, with $\tau=\tau_n$ and the nonlinearity treated as a forcing term, gives
    \begin{align*}
        \EE\Big[\sup_{r \in [0, \tau_n)}\|\omega_r\|^p_{B^{-\frac{2}{p}}_{g,2,p}}\Big]  &+ \EE\Big[\int_0^{\tau_n} \|\omega_r\|_{H^0_g}^p \mathrm{d}r\Big]\\ &\lesssim_\nu \EE\big[\|\omega_0\|^p_{B^{-\frac{2}{p}}_{g,2,p}}\big]+ \expt{\int_0^{\tau_n} \|\phi_R(\|\omega_r\|_{H^r_g}) \LL_{u_r}\omega_r\|^p_{H^{-2}_g} \mathrm{d}r} \\
        &\lesssim_\nu  \EE\big[\|\omega_0\|^p_{B^{-\frac{2}{p}}_{g,2,p}}\big]+ (R+1)^{p(1+k)} \expt{\int_0^{\tau_n} \|\omega_r\|^{p(1-k)}_{H^0_g}\mathrm{d}r } \\
        &\lesssim_\nu\big(1+ \EE\big[\|\omega_0\|^p_{B^{-\frac{2}{p}}_{g,2,p}}\big] \big) + \frac{1}{2} \expt{\int_0^{\tau_n} \|\omega_r\|^{p}_{H^0_g} \mathrm{d}r}. 
    \end{align*}
    We employed \autoref{lem:nonlinear_est_cutoff} with $\delta =0$ and the corresponding $k \in (0,1)$ in the third line. Hence, absorbing the term $\frac{1}{2} \expt{\int_0^{\tau_n} \|\omega_r\|^{p}_{H^0_g} \mathrm{d}r}$ on the left-hand side of the above estimate, we get
      \begin{align*}
        \EE \Big[\sup_{r \in [0, \tau_n)}\|\omega_r\|^p_{B^{-\frac{2}{p}}_{g,2,p}}\Big]  &+ \EE\Big[\int_0^{\tau_n} \|\omega_r\|_{H^0_g}^p \mathrm{d}r\Big]  \lesssim_\nu  1+ \EE\big[\|\omega_0\|^p_{B^{-\frac{2}{p}}_{g,2,p}}\big] .
        \end{align*}
  Letting $n\to\infty$, Fatou's lemma and the blow-up
criterion yield $\PP(\eta<t)=0$. Since $t>0$ is arbitrary,
$\eta=\infty$ almost surely. \\
    
We end this subsection by providing further properties of the global solutions of \eqref{main_ito_cutoff} provided by \autoref{prop:global_well_cutoff}. This is the content of \autoref{lem: H-1 estimate_local_nonlinear} below. Before presenting it, we start with a preliminary computation.

\begin{lemma}\label{lem:noise_estimate}
    Let $\omega \in H_g^0$, and let $\theta$ satisfy \autoref{hyp_noise}. Then it holds that
    \begin{align}
    \label{eq:lie_noise_estimate_without_smoothness}
        \sum_{k\in I} \|\LL_{\sigma_k} \omega\|_{H^{-1}}^2 \leq 3 \|\omega \|_{L^2}^2.
    \end{align}
\end{lemma}
\begin{proof}
    Recall the following equivalence of norms: for every $X\in L^2$, it holds 
    \begin{align*}\|(-\Delta)^{-1/2}\nabla \times X\|_{L^2}= \|\Pi QX\|_{L^2}\end{align*}
   where $Q$ is the orthogonal projector onto mean-zero functions. Then, we have 
    \begin{align*}
        \|\LL_{\sigma_k} \omega \|^2_{ H^{-1}}&= \|(-\Delta)^{-1/2} \nabla \times (\sigma_k\times \omega)\|_{L^2}^2 \\ &= \|(-\Delta)^{-1/2} \nabla \times (\Pi Q[\sigma_k(t)\times \omega])\|_{L^2}^2 \\ &= \|\Pi Q[\sigma_k\times \omega]\|_{L^2}^2 \\ &\leq \|\sigma_k\times \omega\|_{L^2}^2,
        \end{align*}
    and therefore
    \begin{align*}
\sum_{k\in I}\|\sigma_k \times \omega\|_{L^2}^2 &\leq \sum_{k\in I}\|\sigma_k\|_{L^\infty}^2\|\omega\|_{L^2}^2 = \Big(\sum_{k\in I}|\theta_k|^2\Big)\|\omega\|_{L^2}^2 \\
&= 2\Big( \sum_{m\in 2\pi\Z^3_0}|\theta_m|^2\Big)  \|\omega\|^2_{L^2}=3 \|\omega\|^2_{L^2},
\end{align*}
as desired.
\end{proof}

\begin{remark}[Lack of energy conservation in the transport-stretching noise]
\label{r:lack_energy}
The constant $3$ in the estimate of \autoref{lem:noise_estimate} is not optimal. However, one cannot expect the estimate  \eqref{eq:lie_noise_estimate_without_smoothness} to hold with constant $1$ uniformly over $\theta$ satisfying \autoref{hyp_noise}. Indeed, one can check that, taking $a\in(0,1/2)$, 
$
\lambda_n=2\pi\left\lfloor\frac{an}{2\pi}\right\rfloor$
and 
$
\omega_n=e_1\cos(\lambda_n (x\cdot e_3))$, it holds that 
\begin{align*}
\lim_{n\to\infty}
\frac{\sum_{k\in I}
\|\LL_{\sigma_k^n}\omega_n\|_{H^{-1}}^2}
{\|\omega_n\|^2}
=1+c(a)\quad \text{ for some } \quad c(a)>0.
\end{align*}
Thus, as $\|K [\omega]\|= \|\omega\|_{H^{-1}}$ provided $\nabla \cdot \omega=0$, the previous fact implies that one cannot prove a uniform-in-$\theta$ estimate in $H^{-1}$ for $\omega$ (hence, $L^2$ for the velocity) by using variational methods without exploiting hyperviscosity (see e.g., \cite[Chapter 4]{LR15}). To see this, note that the It\^o-Stratonovich corrector $-\nu \Delta$ associated to the Lie-noise \eqref{main_ito} satisfies
$
\langle \Delta \omega,\omega\rangle_{H^{-1}}= -\|\omega\|_{L^2}^2,
$
and cannot alone uniformly compensate-in-$n$ the term \eqref{eq:lie_noise_estimate_without_smoothness} which arises as the It\^o-correction in the corresponding It\^o's formula. However, using the hyperviscosity of the leading operator $A$, one can improve the above situation by replacing $L^2$ with $H^0_g$; see \autoref{lem:high_mode_energy_defect} below.
\end{remark}

\begin{lemma}\label{lem: H-1 estimate_local_nonlinear}
    Let $\omega$ be the global solution provided by \autoref{prop:global_well_cutoff} and $T>0$. For every $p\geq1$ there exists a constant $K(T, \nu, p)$ such that it holds
    \begin{align}
        \EE \Big[\sup_{t\in [0, T]} \|\omega_t\|^{2p}_{H^{-1}} \Big] +\EE \Big[\int_0^T \|\omega_s\|_{ H^{-1}}^{2p-2}\|\omega_s\|_{H^{0}_g}^{2}\mathrm{d}s\Big] \leq K\EE \big[\|\omega_0\|_{ H^{-1}}^{2p}\big]
    \end{align}
\end{lemma}
\begin{proof}(sketch.)
    Applying the finite-dimensional It\^o formula on an appropriate Galerkin approximation of $\omega_t$ and then passing to the limit yields
    \begin{align*}
    \frac12\mathrm{d}\|\omega\|_{ H^{-1}}^2 + (\|\omega\|_{ H^{0}_g}^2 + \nu\|\omega\|_{L^2}^2)\mathrm{d}t= \sqrt{\nu}\mathrm{d}M_t + \nu\sum_{k\in I} \|\LL_{\sigma_k} \omega\|_{H^{-1}}^2\mathrm{d}t,
\end{align*}
 where $\mathrm{d}M_t = \sum_{k\in I} \brak{(-\Delta^{-1})\nabla \times (\sigma_k \times \omega_t), \omega_t}\mathrm{d}W_t^k$.
 An easy interpolation argument (similar to the one in the proof of \autoref{SMR_linear_prop}) gives, for every $\eps>0$, a constant $C_\eps$ such that 
\begin{equation}\label{L2_interpolation}
    \|\omega_s\|_{L^2}^2 \leq \eps \|\omega_s\|_{H_g^0}^2 + C_\eps \|\omega_s\|^2_{H^{-1}}.
    \end{equation}
 This, together with \autoref{lem:noise_estimate} gives 
 \begin{align}\label{ito_H-1}
    \frac12\mathrm{d}\|\omega\|_{ H^{-1}}^2 + \nu\|\omega\|_{L^2}^2\mathrm{d}t+ \frac{1}{2}\|\omega\|_{H^{0}_g}^2\mathrm{d}t \leq  \sqrt{\nu}\mathrm{d}M_t + C_\nu\| \omega\|_{H^{-1}}^2\mathrm{d}t.
\end{align}

We deduce 
\begin{align*}\sup_{s \in [0, T]}\expt[]{\|\omega_s\|_{ H^{-1}}^2} + \expt[]{\int_0^T \|\omega_s\|_{H^{0}_g}^2\mathrm{d}s} \leq K_\nu\expt[]{\|\omega_0\|_{H^{-1}}^2}.\end{align*}
Next, using Burkholder-Davis-Gundy and H\"older inequalities we get 
\begin{align*}
\EE \Big[\sup_{s\in [0,T]}|M_s|\Big]&\lesssim  \expt[]{\left(\sum_{k\in I}\int_0^T\|(-\Delta)^{-1/2}\LL_{\sigma_k(s)}\omega_s\|^2_{L^2}\|\omega_s\|^2_{ H^{-1}}\mathrm{d}s\right)^{1/2}} \\ &\lesssim  \expt[]{\sup_{s\in [0, T]} \|\omega_s\|_{ H^{-1}} \left(\int_0^T \|\omega_s\|^2_{L^2}\mathrm{d}s\right)^{1/2}} \\
&\leq \frac{1}{4\sqrt \nu}\expt[]{\sup_{s\in [0, T]} \|\omega_s\|_{H^{-1}}^2 } + C_\nu\expt[]{\int_0^T \|\omega_s\|^2_{L^2}\mathrm{d}s}\\
&\leq \frac{1}{4\sqrt{\nu}}\expt[]{\sup_{s\in [0, T]} \|\omega_s\|_{H^{-1}}^2 } + \tilde C_\nu\expt[]{\int_0^T \|\omega_s\|^2_{H^{-1}}\mathrm{d}s} + \eps \expt[]{\int_0^T \|\omega_s\|^2_{H^0_g}\mathrm{d}s},
\end{align*}
where we used again \eqref{L2_interpolation} in the last step. 
Taking the supremum in time of the time-integrated equation \eqref{ito_H-1} and using the last inequality for the martingale term, we get 
\begin{equation}\label{energy_estimate_H-1}
    \EE \Big[\sup_{s\in [0, T]}\|\omega_s\|_{ H^{-1}}^2\Big] + \EE \Big[\int_0^T \|\omega_s\|_{ H^{0}_g}^2\mathrm{d}s \Big]\leq K_\nu\EE \big[\|\omega_0\|_{ H^{-1}}^2\big].
\end{equation}
For the case $p>1$ we apply again It\^o formula to $\|\omega\|_{H^{-1}}^{2p}$ obtaining 
\begin{align*}
    \mathrm{d}\|\omega\|_{ H^{-1}}^{2p} &= 2p \|\omega\|_{ H^{-1}}^{2p-2}\left( -\nu \|\omega\|_{L^2}^2 - \|\omega\|_{ H^0_g}^2\right)\mathrm{d}t  \\&\quad + \nu2p(2p-2)\|\omega\|_{ H^{-1}}^{2p-4}\sum_{k\in I}\left|\brak{\LL_{\sigma_k}\omega, \omega}_{ H^{-1}}\right|^2\mathrm{d}t \\
    &\quad+ \nu2p\|\omega\|_{H^{-1}}^{2p-2}\sum_{k\in I}\|\LL_{\sigma_k}\omega\|_{ H^{-1}}^2\mathrm{d}t +\mathrm{d}M_t\\
    &\leq 2p \|\omega\|_{ H^{-1}}^{2p-2}\left( -\nu \|\omega\|_{L^2}^2 - \|\omega\|_{ H^0_g}^2 + 3\nu \|\omega\|_{L^2}^2\right)\mathrm{d}t \\ & \quad + \nu 2p(2p-2)\|\omega\|_{H^{-1}}^{2p-4}\sum_{k\in I}\left|\brak{\LL_{\sigma_k}\omega, \omega}_{ H^{-1}}\right|^2\mathrm{d}t + \mathrm{d}M_t \\
    &\leq 2p \|\omega\|_{ H^{-1}}^{2p-2}\left( - \|\omega\|_{H^0_g}^2 + 2\nu (3p-2)\|\omega\|_{L^2}^2\right)\mathrm{d}t + \mathrm{d}M_t,
\end{align*}
where $M$ now denotes the corresponding continuous local martingale. The formula above implies the desired estimate by the same steps as in the previous case.
\end{proof}

\section{Scaling Limit for a system with cutoff}\label{sec:scaling_limit}
We begin by setting
\begin{equation*}
    \B_p(M)\coloneq  \Big\{\xi \in B^{-2/p}_{g,2,p} : \ \  \|\xi\|_{ B^{-2/p}_{g,2,p}}\leq M \Big\}.
\end{equation*}
The goal of this section is to prove the following result:
\begin{theorem}\label{thm: main scaling}
    Let $T>0,$ $\nu>0$, $R> 0$, $M\geq 1$ and $\theta^n$ satisfying \autoref{hyp_noise}. Assume $p \in (4, \infty)$, $-1/2 < r<- 2/p$. Let $\omega^n$ be the solution of \eqref{main_ito_cutoff} associated to $\theta^n$ and some initial data $\omega_0\in \B_p(M)$. Let $\omega^{R,\det}$ be the solution of 
    \begin{equation}\label{main_det_cutoff}
\begin{cases}
    \partial_t\omega^{R,\det} +  \phi_{R}(\|\omega^{R,\det}\|_{H^r_g})\LL_u \omega^{R,\det} - A\omega^{R,\det} - \nu\Delta \omega^{R,\det}   =0 \\
    u= K[ \omega^{R,\det}],
    \end{cases}
\end{equation}
with the same initial datum $\omega_0$ (cf. \autoref{rem:stoch2det}). Then for all $r_0 < -2/p$ and $\eps >0$ we have 
    \begin{equation}
        \lim_{n\rightarrow\infty}\sup_{\omega_0 \in \B_p(M)} \PP\Big( \sup_{t\in [0, T]}\|\omega^n_t - \omega^{R,\det}_t\|_{H^{r_0}_g} > \eps \Big) =0.
    \end{equation}
\end{theorem}
To prove the above theorem, we will need the following time estimates.
\begin{proposition}\label{lem: time estimates}
    Let $p>4$, and $-1/2< r < -2/p$ and consider 
    \begin{align*}\M^n_t\coloneq  \sum_{k\in I }\int_0^t \LL_{\sigma^n_k}\omega^n_s \mathrm{d}W_s^k.\end{align*}
    Then for every $T>0, a>1$ there exists $r_1, K_1>0, \ s_1\in [0, 1/2)$ such that 
    \begin{align}
        \expt[]{\|\M^n\|^{2a}_{C^{s_1}(0,T; H_g^{-r_1})}} &\leq K_1 \EE[\|\omega_0\|^{2a}_{H^{-1}}] \label{time estimate mart}, \\
         \expt[]{\|\omega^n\|^{2a}_{C^{s_1}(0,T; H_g^{-r_1})}} &\leq K_1 \EE[\|\omega_0\|^{2a}_{H^{-1}}] \label{time estimate}.
    \end{align}
\end{proposition}

\begin{proof}
    First, we show that $\M^n_t\in W^{s,q}(0, T;  H^{-\rho}_g)$ for some $s,q>0$ and $\rho> 2$, and then we obtain the statement by a usual embedding theorem. 
    Since $\rho >1$ we have the embedding $H^{-\rho}_{\operatorname{div}}\hookrightarrow H^{-\rho}_g$, we have, by BDG 
    \begin{align*}
        \expt[]{\|\M^n_t-\M^n_s\|_{ H^{-\rho}_g}^{2q}} &\lesssim \expt[]{\left(\sum_{k\in I}\int_s^t\|\LL_{\sigma^n_k}\omega^n_r\|_{ H^{-\rho}}^{2}\mathrm{d}r\right)^q} \\
        &\lesssim  \expt[]{\left(\sum_{k\in I}\int_s^t\|\Pi Q(\sigma^n_k \times \omega^n_r)\|_{H^{-\rho +1}}^{2}\mathrm{d}r\right)^q}.
    \end{align*}
     Now we have, for $k=(m,j)\in I$ and $n\leq |m|\leq 2n$
    \begin{equation*}
        Q(\sigma^n_k  \times \omega^n_r)= Z_n^{-1}\sum_{\substack{(l,i)\in I\\ l\neq -m}} |m|^{-3/2}\omega^{l,i,n}_r a_{m,j} \times a_{l,i}e^{i(m+l)\cdot x},
    \end{equation*}
    where for each $(\alpha,\beta)\in I$
    \begin{align*}
    \omega^{\alpha,\beta,n}_r=\langle \omega^n, a_{\alpha,\beta}e^{i \alpha\cdot x}\rangle.  
    \end{align*}
    Therefore, since $\rho >1$, 
    \begin{align*}
        \sum_{\substack{k\in I \\ k=(m,j)}}\|\Pi Q(\sigma_k \times \omega_r)\|_{H^{-\rho +1}}^{2} &\leq Z_n^{-2 } \sum_{\substack{(m,j) \in I \\ n\leq |m| \leq 2n}}|m|^{-3}\sum_{\substack{(l,i)\in I\\ l\neq -m }} |l+m|^{-2\rho+2}|\omega^{l,i,n}_r|^2\\
        &\leq2 Z_n^{-2 } \sum_{\substack{(l,i)\in I}}\frac{|\omega^{l,i,n}_r|^2}{|l|^2} \left(\sum_{\substack{m \in 2\pi\Z_0^3 \\ n\leq |m| \leq 2n \\ l\neq -m }}|m|^{-3} |l|^2|l+m|^{-2\rho+2}\right)\\
        &\lesssim  \|\omega^n_r\|^2_{H^{-1}}
        \end{align*}
        where the last inequality holds provided that the following series is convergent and bounded uniformly for $l\in2\pi\Z_0^3$
        \begin{align*}
            \sum_{\substack{m \in 2\pi\Z_0^3 \\ n\leq |m| \leq 2n \\ l\neq -m }}|m|^{-3} |l|^2|l+m|^{-2\rho+2}.
        \end{align*}
        This fact is proved in \cite[Lemma 4.5]{butori_meanfield_2025}, for any $\rho>2$. The key observation in the proof is that the bad part of this sum consists of the terms for which $|l+m|\leq |l|/2$, for which we use that $|m| \geq |l|/2$ to reduce to the convergent series $\sum_{n\leq |m| \leq 2n} |m|^{-1}|l+m|^{-2\rho + 2} \lesssim 1$. 
        
           This inequality implies that
           \begin{align*} \expt[]{\|\M^n_t-\M^n_s\|_{ H^{-\rho}}^{2q}}\lesssim |t-s|^q \EE\left[\|\omega^n\|_{L^\infty_t H_x^{-1}}^{2q}\right]\lesssim |t-s|^q \EE\left[\|\omega_0\|_{ H^{-1}}^{2q}\right],\end{align*}
           where the control on the $H^{-1}$ norm follows from \autoref{lem: H-1 estimate_local_nonlinear}. Consequently, for $\alpha < 1/2 $
\begin{align*}\int_0^T\int_0^T \frac{ \expt[]{\left\|\M^n_t- \M^n_s \right\|_{H^{-\rho}_g}^{2q}}}{|t-s|^{1+2\alpha q}}\mathrm{d}s\mathrm{d}t \lesssim 1.\end{align*}
 This estimate implies $\M^n_t\in W^{\alpha,2q}(0, T;  H_g^{-\rho})$, and provided $2\alpha q >1, \alpha< 1/2 $, which is true for instance for $q= 2$ and $\alpha = 1/3$, this space embeds in a H\"older space. 
 
The proof of the second bound follows easily from this, as in \cite[Lemma 3.4]{flandoli_high_2021}. 
\end{proof}

We are now in a position to prove the main result of the section.

\begin{proof}[Proof of \autoref{thm: main scaling}]
    Thanks to the uniform estimates provided by \autoref{prop:global_well_cutoff} and \autoref{lem: time estimates}, and the compact inclusion
    \begin{align*}C(0, T; B^{-\frac{2}{p}}_{g,2,p}) \cap C^{s_1}(0, T; H^{-r_1}_g)\underset{c}{\hookrightarrow} C(0, T; H_g^{r_0})\end{align*}
   with $-1/2<r<r_0 < -2/p$, we can apply Prokhorov's theorem to the sequence $\mu^n$ of laws of $\omega^n$, obtaining a weak limit $\mu$ on $C(0, T; H^{r_0}_g)$.  
   Next by Skhorokhod's representation theorem, up to passing to subsequences we can find a new probability space that for simplicity we continue to denote $(\Omega, \F, \F_t, \PP)$ and processes \begin{align*}\big(\tilde \omega^n, W^n= (W_t^{k,n})_{k\in I}\big), \quad \big(\tilde \omega, W= ( W_t^k)_{k\in I}\big)\end{align*} such that $\tilde \omega^n \rightarrow \tilde \omega$ in $C([0, T]; H^{r_0}_g),\ \tilde{u}^n\coloneq K[ \tilde{\omega}^n]\rightarrow K[ \tilde{\omega}]\eqcolon\tilde{u}$ in $C([0,T];L^2)$ and $ W^n \rightarrow  W$  in $C(0, T; \C^I)$ $\PP-$almost surely (in the new probability space) and the processes $\tilde\omega^n$ are still weak solutions of equation \eqref{main_ito_cutoff}. Moreover, thanks to \autoref{prop:global_well_cutoff}, we also have, up to further subsequences, 
   \begin{align*}
       \tilde \omega^n \rightharpoonup \tilde \omega \quad \text{in } L^p(\Omega\times (0, T); H^0_g).
   \end{align*}
   It follows that for every $\psi \in C^\infty(\T^3;\R^3)$, we have
   \begin{align*} \sup_{t\in [0,T]} |\brak{\tilde \omega^n_t - \tilde \omega_t, \psi}| + \sup_{t\in [0,T]} \left|\int_0^t \brak{\tilde \omega^n_s - \tilde \omega_s, (A + \nu \Delta)\psi}\mathrm{d}s\right| \rightarrow 0. \end{align*}
   Next,
\begin{align*}
    \sup_{t\in[0,T]}
    &\left|\int_0^t
    \brak{
        \phi_R(\|\tilde\omega_s^n\|_{H_g^r})
        \LL_{\tilde u_s^n}\tilde\omega_s^n
        -
        \phi_R(\|\tilde\omega_s\|_{H_g^r})
        \LL_{\tilde u_s}\tilde\omega_s,
        \psi
    }\,\mathrm{d}s\right| \\
    &\leq \int_0^T
    \phi_R(\|\tilde\omega_t^n\|_{H_g^r})
    \left|\brak{
        \tilde u_t^n\otimes\tilde u_t^n
        -\tilde u_t\otimes\tilde u_t,
        \nabla(\nabla\times\psi)
    }\right|\,\mathrm{d}t \\
    &+\int_0^T
    \left|\phi_R(\|\tilde\omega_t^n\|_{H_g^r})
        -\phi_R(\|\tilde\omega_t\|_{H_g^r})\right|
    \left|\brak{
        \tilde u_t\otimes\tilde u_t,
        \nabla(\nabla\times\psi)
    }\right|\,\mathrm{d}t \\
    &\lesssim \|\psi\|_{C^2}
    \int_0^T
    \left(\|\tilde u_t^n\|_{L^2}
        +\|\tilde u_t\|_{L^2}\right)
    \|\tilde u_t^n-\tilde u_t\|_{L^2}\,\mathrm{d}t \\
    &+\|\psi\|_{C^2}
    \sup_{t\in[0,T]}
    \|\tilde\omega_t^n-\tilde\omega_t\|_{H_g^r}
    \int_0^T\|\tilde u_t\|_{L^2}^2\,\mathrm{d}t \\
    &\rightarrow 0\quad\mathbb{P}-a.s.
\end{align*}
        Finally, thanks to the Burkholder-Davis-Gundy inequality
        \begin{align*}
            \expt{\sup_{t\in [0, T]}\left(\sum_{k \in I} \int_0^t \brak{\sigma^n_k \times \tilde \omega^n_s , \nabla \times \psi} \mathrm{d}W_s^{k,n}\right)^2} &\lesssim  \expt{ \sum_{k \in I} \int_0^T \left|\brak{\sigma^n_k , \tilde \omega^n_s \times (\nabla \times \psi)}\right|^2 \mathrm{d}s}  \\
            &\lesssim \|\theta^n_k\|^2_{l^\infty}\|\psi\|_{C^1}^2 \expt{\int_0^T\|\tilde \omega^n_s \|_{L^2}^2\mathrm{d}s} \rightarrow 0,
        \end{align*}
        since clearly $L^p(\Omega\times (0, T); H^0_g) \subset L^2(\Omega\times (0, T); L^2)$. Thus, up to another subsequence 
        \begin{align*}
            \sup_{t\in [0, T]}\left|\sum_{k \in I} \int_0^t \brak{\sigma^n_k \times \tilde \omega^n_s , \nabla \times \psi} \mathrm{d}W_s^{k,n}\right| \rightarrow 0 \quad \PP- a.s.
        \end{align*}
        It follows that $\tilde \omega \in C(0, T; H^{r_0}_g) \cap L^p(0, T; H^0_g)$ is a weak solution of \eqref{main_det_cutoff} in the sense that for every $\psi $ zero-mean, divergence free and smooth it holds 
        \begin{align}
            \brak{\tilde\omega_t, \psi}- \brak{\tilde\omega_0, \psi}= \int_0^t \brak{\tilde{\omega}_s, (A+ \nu \Delta )\psi}+\phi_R(\|\tilde\omega_s\|_{H^r_g})\brak{\tilde u_s \times \tilde \omega_s, \nabla \times \psi}\mathrm{d}s.
        \end{align}
        The same argument as in \cite[Lemma 5.6]{agresti_global_2024} shows that $\tilde\omega \in W^{1,p}(0, T; H^{-2}_g) \cap L^p(0, T; H^0_g) \subset C([0, T]; B^{-2/p}_{g,2,p})$, and thus it is also a $p$-solution in the sense analogous to \autoref{def:local_well_posed}. Hence, by uniqueness, cf. \autoref{prop:global_well_cutoff}, it coincides with the unique $p$-solution.
        It follows that $\mu^n \rightarrow \delta_{\omega^{R,\det}}$. Thus, since the limit measure is a delta, the whole sequence $\omega^n$ converges in probability to $\omega^{R,\det}$, in the original probability space.
        We are left to prove the uniformity with respect to $\omega_0 \in \B_p(M)$. First we notice that the proof holds in the same way if we let the initial condition for the stochastic problem depend on $n$ with $\omega_0^n \rightharpoonup \omega^0$ in $B^{-\frac{2}{p}}_{g,2,p}$ and we replace $\omega^{R,\det}$ with $\omega^{R,0,\det}$ given by the solution of \eqref{main_det_cutoff} with initial data $\omega^0$. Then, if by contradiction there existed a sequence of initial data $\{\omega_0^n\}_{n\geq 1} \in \B_p(M)$ and $\delta>0$ such that for every $n\geq 1$ 
         \begin{equation*}
             \PP\Big( \sup_{t\in [0, T]}\|\omega^n_t - \omega^{R,n,\det}_t\|_{H^{r_0}_g} > \eps \Big) \geq \delta,
         \end{equation*}
         then, up to a subsequence, we could assume that $\omega_0^n \rightharpoonup \omega^0$. This implies
         \begin{align*}
             \PP\Big( \sup_{t\in [0, T]}\|\omega^n_t - \omega^{R,0,\det}_{t}\|_{H^{r_0}_g} > \frac{\eps}{2} \Big) \rightarrow 0.
         \end{align*}
         The latter gives a contradiction, since it is easy to show that $\|\omega^{R,n\det} - \omega^{R,0\det}\|_{L^\infty(0, T; H^{r_0}_g)} \rightarrow 0$. 
\end{proof}

\section{Global smooth solutions -- Proof of \autoref{thm:main}}
\label{sec:proof_of_main}
In light of \autoref{thm: main scaling}, as known in the context of scaling limits (see e.g., \cite{agresti_global_2024,FGL21_delayed,flandoli_high_2021}), the proof of \autoref{thm:main} now follows by combining two ingredients:
\begin{itemize}
\item 
Small data implies global solutions to the Navier-Stokes equations with Lie noise transport uniformly in $\theta$ as in \autoref{hyp_noise} -- Proposition \ref{prop:small_implies_global}.
\item Global smooth solutions with exponential decay of deterministic Navier-Stokes equations with high viscosity -- \autoref{lem:dissipation}.
\end{itemize}

The proofs of the above results are presented in the following subsections.
Let us emphasize that, since the Lie noise transport used in \eqref{main_ito} is not energy-preserving, the usual small-implies-global strategy in \cite{agresti_global_2024,flandoli_high_2021} does not extend. For the latter, we fully exploit the super-quadratic behaviour of the Fourier symbol associated with the operator $A$ (see \autoref{Defintition_operator}).

\subsection{Global smooth solutions of Navier-Stokes equations with Lie noise and small data}
The main result of this subsection reads as follows.

\begin{proposition}[Small data implies global existence uniformly in $\theta$]
\label{prop:small_implies_global}
Let $4\leq p<+\infty$. Then, there exists $n_0$ sufficiently large depending only on $\nu$ and $p$ for which the following assertions hold.
For all $\varepsilon\in (0,1)$, there exists a constant $\eta\in (0,1)$ such that, for all $s\geq 0$, $n\geq n_0$, $\theta=\theta^n$ as in \autoref{hyp_noise}, and $\omega_s\in L^0_{\F_s}(\Omega;B_{g,2,p}^{-2/p})$, it holds that 
\begin{align*}
\EE\big[\|\omega_s\|_{B^{-2/p}_{g,2,p}}^p\big]\leq \eta \qquad \Longrightarrow \qquad \PP(\tau=\infty)>1-\varepsilon
\end{align*}
where $\tau$ is the lifetime of the solution $(\omega,\tau)$ to \eqref{main_ito} provided by \autoref{prop:local_well_nonlinear}. Moreover, there exists a stopping time $\mu\leq \tau$ such that $\PP(\mu=\infty)>1-\varepsilon$ and a constant $C>0$ independent of $\omega_s$, $\theta$ and $\varepsilon$ such that 
\begin{equation}
\label{eq:prop:small_implies_global_bound}
\EE\Big[ \sup_{t\in[ s,\mu)} \|\omega_t\|_{B^{-2/p}_{g,2,p}}^p \Big]
+\EE \int_s^{\mu}\|\omega_t\|_{H^{0}_{g}}^p\,\dd t 
\leq C \big(1+
\EE\big[\|\omega_s\|_{B^{-2/p}_{g,2,p}}^p  \big]\big).
\end{equation}
\end{proposition}

The proof of the above result follows the line of 
\cite[Theorem 4.3]{agresti_global_2024}, and is given at the end of this subsection. However, in contrast to the latter, we cannot use energy methods to absorb lower-order terms produced by the perturbation argument in the stochastic maximal regularity estimates in \autoref{SMR_linear_prop} (see also \autoref{SMR_perturbation_estimate}).
To overcome this, we exploit the growth of the Fourier symbol associated with the operator $A$ (see \autoref{Defintition_operator}) via the following result.

\begin{lemma}\label{lem:high_mode_energy_defect}
Let $g$ and $\theta$ be as in \autoref{hp_correction} and \autoref{hyp_noise}, respectively. 
There exists a constant $C$ such that, for all $2\pi\leq \Lambda<n/2$ and $v\in H_g^0$, it holds that 
\begin{align}
 \Big|\sum_{k\in I}\|\LL_{\sigma_k^n}v\|_{H^{-1}}^2-\|v\|_{L^2}^2\Big|
 &\leq C\Big(\frac{\Lambda}{n}+\frac1{g(\Lambda)}\Big)\|v\|_{H_g^0}^2,
 \label{eq:high_mode_energy_defect}\\
 \sum_{k\in I}|\langle\LL_{\sigma_k^n}v,v\rangle_{H^{-1}_{\mathrm{div}}}|^2
 &\leq \frac{C}{g(\Lambda)}\|v\|_{H^{-1}_{\mathrm{div}}}^2\|v\|_{H_g^0}^2.
 \label{eq:high_mode_martingale_defect}
\end{align}
\end{lemma}

Choosing, for instance, $\Lambda=\sqrt{n}\to \infty$, the constants on the right-hand side of \eqref{eq:high_mode_energy_defect}-\eqref{eq:high_mode_martingale_defect} become arbitrarily small. 
In particular, when computing the energy, the above shows that the contribution of the It\^o-Stratonovich corrector can be made arbitrarily close to It\^o's correction contribution up to a small factor in the $H^0_g$-norm; see \autoref{r:lack_energy}.
This is the key to show a long-time version of \autoref{SMR_linear_prop} proven below. To avoid interrupting the flow of the presentation, we postpone the proof of \autoref{lem:high_mode_energy_defect} to \autoref{app_techical_lemma}.

\begin{lemma}[Uniform maximal regularity on the half-line]\label{SMR_halfline_prop}
Let $\nu>0$, $p\in (2,\infty)$, and let $\theta^n$ be as in \autoref{hyp_noise}. There exist $n_0$ and $C$ depending only on $p$ and $\nu$ such that, for all $n\geq n_0$, time $s\geq0$, every stopping time $\sigma\geq s$, initial data $u_s\in L^0_{\F_s}(\Omega;B^{-2/p}_{g,2,p})$, and progressively measurable $f\in L^p((s,\sigma)\times \Omega;H^{-2}_g)$, there is a local $p$-solution to (cf. \autoref{def:local_well_posed})
\begin{equation}\label{linear_halfline_equation}
 \mathrm{d}u-(A+\nu\Delta)u\,\mathrm{d}t
 =f\,\mathrm{d}t-\sqrt\nu\sum_{k\in I}\LL_{\sigma_k^n}u\,\mathrm{d}W_t^k,
 \qquad u(s)=u_s,
\end{equation}
satisfying
\begin{equation}\label{smr_halfline_inequality}
 \EE\Big[\sup_{s\leq t<\sigma}\|u(t)\|_{B^{-2/p}_{g,2,p}}^{p}\Big]
 +\EE\Big[\int_s^\sigma\|u(t)\|_{H_g^0}^{p}\,\mathrm{d}t\Big]
 \leq C\Big(\EE\big[\|u_s\|_{B^{-2/p}_{g,2,p}}^{p}\big]
+\EE\Big[\int_s^\sigma\|f(t)\|_{H_g^{-2}}^{p}\,\mathrm{d}t\Big]\Big).
\end{equation}
\end{lemma}

Compared to \autoref{SMR_maximal_thm}, in the above we allow for stopping times $\sigma$ satisfying $\PP(\sigma=\infty)>0$.

\begin{proof}
By translation, we may assume $s=0$. Existence is already proven in \autoref{SMR_maximal_thm}. The uniform estimate \eqref{smr_halfline_inequality} now follows similarly to \cite[Theorem 3.1(2)]{agresti_global_2024} by using the exponential stability of the heat semigroup on Sobolev classes of mean-free vector fields. 
Here, we content ourselves with explaining how to bound lower-order terms of the form 
\begin{align*}
    \EE\Big[\int_0^\sigma\|u\|_{H^{-1}_{\mathrm{div}}}^p\,\mathrm{d}t\Big]
\end{align*}
uniformly in $\theta$ and $\sigma:\Omega\to [0,\infty]$. To do the above, we apply the variational method. Applying \autoref{lem:high_mode_energy_defect} with $\Lambda=\sqrt{n}$, it follows that there exists $n_0$ such that the estimates in \autoref{lem:high_mode_energy_defect} hold with arbitrarily small constants. 
Thus, the It\^o's formula applied to $\omega\mapsto \|\omega\|_{H^{-1}}^p$ (see e.g., \cite[Lemma 3.1]{GHV24}) readily yields, using \eqref{eq:high_mode_martingale_defect} to control the additional second-order term arising from It\^o’s formula,
\begin{align*}
 \mathrm{d}\|u\|_{H^{-1}}^p
 +\frac{p}{2}\|u\|_{H^{-1}}^{p-2}\|u\|_{H^0_g}^2\,\mathrm{d}t
 &\leq p\|u\|_{H^{-1}}^{p-2}\langle f,u\rangle_{H^{-2}_g,H^0_g}\,\mathrm{d}t
 +\mathrm{d}M_t\\
 &\leq \frac14\|u\|_{H^{-1}}^{p-2}\|u\|_{H^0_g}^2\dd t
 +C_{p,g}\|f\|_{H^{-2}_g}^p\dd t+\mathrm{d}M_t.
\end{align*}
where $M_t$ is a continuous martingale starting at zero. Taking expectations and using $H^{0}_g \hookrightarrow H^{-1}$, we obtain 
\begin{equation}\label{forced_halfline_energy}
 \EE\Big[\int_0^\sigma\|u\|_{H^{-1}}^p\,\mathrm{d}t\Big]
 \lesssim_{p,g}\EE\Big[\|u_0\|_{H^{-1}}^p\Big]
 +\EE\Big[\int_0^\sigma\|f\|_{H^{-2}_g}^p\,\mathrm{d}t\Big].
\end{equation}
The rest of the proof follows as in \cite[Theorem 3.1(2)]{agresti_global_2024}.
\end{proof}

\begin{proof}[Proof of \autoref{prop:small_implies_global}]
In light of \autoref{SMR_halfline_prop}, the proof follows almost verbatim from the one in \cite[Theorem 4.3]{agresti_global_2024} just by replacing \cite[equation (4.13)]{agresti_global_2024} with 
\begin{align*}
     \EE\Big[\int_s^\sigma\|\omega_t\|_{H^{-1}}^p\,\mathrm{d}t\Big]\lesssim_{p,g} \mathbb{E}\big[\|\omega_s\|^p_{H^{-1}}\big].
\end{align*}
The latter follows by the variational approach arguing as for the proof of \eqref{forced_halfline_energy} with $f=-\mathcal{L}_u\omega$ and using the cancellation $\langle \mathcal{L}_u\omega, (-\Delta)^{-1}\omega\rangle=0.$
To avoid repetition, we omit details.
\end{proof}

\subsection{Proof of \autoref{thm:main}}
As outlined at the beginning of this section, the last missing ingredient in the proof of \autoref{thm:main} is the following result on Navier-Stokes equations with high dissipation.

\begin{lemma}\label{lem:dissipation}
    Let $M\geq 1$ and $r_0,\gamma$ satisfying $-\frac{1}{2}<r_0<\gamma<0$. There exists $\nu_0=\nu_0(M,\gamma)>0$ such that, for every $\nu>\nu_0$, and every $\omega_0\in H_g^\gamma$ satisfying
$\|\omega_0\|_{H_g^\gamma}\leq M$, there exists a unique global solution $\omega^{\det} $ of 
    \begin{equation}\label{main_det}
\begin{cases}
    \partial_t\omega^{\det} + \LL_u \omega^{\det} - A\omega^{\det} - \nu\Delta \omega^{\det}   =0 \\
    u= K[ \omega^{\det}].
    \end{cases}
\end{equation}
Moreover, there exists $\eta_0=\eta_0(M,\gamma,r_0) >0$, such that 
    \begin{align}
    \sup_{t>0}\|\omega^\det_t\|_{H_g^\gamma}^2 + \|\omega^\det \|_{L^2(0, +\infty; H^{\gamma+1}_g)}^2 &\lesssim \|\omega_0\|_{H_g^\gamma}^2 \label{det_energy_estimates_1},\\
      \sup_{t>0}\left(   e^{\eta_0 t}\|\omega^\det_t\|_{H^{r_0}_g}\right) &\lesssim  \|\omega_0\|_{H^\gamma_g} .\label{det_energy_estimates_2}
    \end{align}
    The implicit constants in the estimates above depend only on
$M$ and $\gamma$.
\end{lemma}
\begin{proof}
To prove the local well-posedness, we employ the following estimate on the nonlinearity 
\begin{equation}\label{eq:nonlinear-gamma}
    \|\nabla\times\operatorname{div}(K[\omega]\otimes K[\xi])\|_{H_g^{\gamma-1}}
    \lesssim
    \|\omega\|_{H_g^{\gamma/2+1/4}}
    \|\xi\|_{H_g^{\gamma/2+1/4}},
\end{equation} 
This is obtained by interpolation (see \cite[Theorem 4.1.1]{MR482275}) showing that the desired inequality holds for $\gamma =-1$ (\emph{cf.} \autoref{lem:nonlinear_est}) and $\gamma = 0$. To prove the case $\gamma =0$, introducing $u=K[\omega],\ v=K[\xi]$, we have
\begin{align*}
\|\nabla\times\operatorname{div}(u\otimes v)\|_{H^{-1}}
\lesssim \|u\|_{H^{5/4}}\|v\|_{H^{5/4}}
\lesssim_g \|\omega\|_{H_g^{1/4}}\|\xi\|_{H_g^{1/4}}.
\end{align*}
We now show how to obtain local well-posedness. It suffices to prove an estimate in $H^\gamma_g$. We have :
        \begin{align*}
            \frac{\mathrm{d}}{\mathrm{d}t}\|\omega^\det\|_{H_g^\gamma}^2 + 2 \nu\|\nabla \omega^\det\|_{H_g^\gamma}^2 + 2\|\omega^\det\|_{H_g^{\gamma + 1}}^2 = -2\brak{(-A)^{\gamma/2} \LL_u \omega^\det, (-A)^{\gamma /2}\omega^\det}_{H_g^0}.
        \end{align*}
        Thanks to \eqref{eq:nonlinear-gamma} we have 
        \begin{align*}\left|\langle(-A)^{\gamma/2} \LL_u \omega^\det, (-A)^{\gamma /2}\omega^\det\rangle_{H_g^0}\right| \lesssim \|\omega^\det\|_{H^{\gamma+1}_g}\|\omega^\det\|^2_{H^{\gamma/2 +1/4}_g} \end{align*}
        In particular, it follows 
        \begin{equation}\label{eq:e.e._gamma}
            \sup_{[0, T]} \|\omega^\det\|_{H^\gamma_g}^2 + \|\omega^\det\|^2_{L^2(0, T; H_g^{\gamma + 1})} \lesssim \|\omega_0\|_{H^\gamma_g}^2 + \|\omega^\det\|^4_{L^4(0, T; H_g^{\gamma/2 + 1/4})}.
             \end{equation}
             Since $\gamma+1/2  > \gamma /2 + 1/4$ thanks to our assumptions, the above allows us to obtain the local well-posedness by standard arguments; we omit the easy details. 
In addition, thanks to \cite[Corollary 2.3(i)]{PRUSS20182028}, we have the blow-up criterion
\begin{equation*}
    \tau < \infty \implies \sup_{t\in [0, \tau)}\|\omega^\det\|_{H^{\gamma}_g} = \infty.
\end{equation*}
Next, we show how \eqref{det_energy_estimates_2} follows from \eqref{det_energy_estimates_1}.
    First, since 
    $\langle\LL_u \omega, (-\Delta)^{-1}\omega\rangle =0$,
    the energy inequality for $\omega_\det$ gives
    \begin{equation}\label{det_nonlinear_en_estimate}
        \|\omega^\det_t\|^2_{H^{-1}} + \nu\int_0^t \|\omega^\det_s\|_{L^2}^2\mathrm{d}s + \int_0^t \|\omega^\det_s\|_{H^{0}_g}^2\mathrm{d}s \leq  \|\omega_0\|^2_{H^{-1}}.
    \end{equation}
   Thus, we get by Gr\"onwall's inequality that for some $\eta_0 >0$, it holds
    $\|\omega^\det_t\|_{H^{-1}} \leq  e^{-\eta_0 t}\|\omega_0\|_{H^{-1}} \leq  e^{-\eta_0 t}\|\omega_0\|_{H^{\gamma}}$ and by interpolation, 
    \begin{equation}
        \|\omega^\det\|_{H^{r_0}_g}\leq e^{-\eta_0 \theta t}\|\omega_0\|^{\theta}_{H_g^{\gamma}}\|\omega^\det\|_{H^{\gamma}_g}^{1-\theta}
    \end{equation}
    This concludes the proof if we can show that for some constant $K_0$, independent of time, it holds 
    \begin{equation}\label{det_goal}
        \sup_{t\geq 0} \|\omega^\det_t\|_{H^\gamma_g} \leq K_0 \|\omega_0\|_{H_g^\gamma}.
         \end{equation}
        Equation \eqref{det_goal} also implies $\tau = +\infty$ by the above blow-up criterion. 
    The key point of the remaining part of the proof is to bound the last term in \eqref{eq:e.e._gamma}. We make use of the freedom to choose $\nu$ as large as we want. 
By interpolation we have \begin{equation}\label{2p-1p estimate}
    \|\omega^\det\|_{L^{4}(0, T; H_g^{\gamma + 1/2})} \leq\|\omega^\det\|^{1/2}_{L^{\infty}(0, T; H^{\gamma}_g)}\|\omega^\det\|^{1/2}_{L^{2}(0, T; H_g^{\gamma+1})}. 
    \end{equation}
    Using Young's inequality and the estimate \eqref{eq:e.e._gamma}, we get, for a hidden constant independent of $\nu$,
\begin{align*}
    \|\omega^\det\|_{L^{4}(0, T; H_g^{\gamma +1/2})} \lesssim \|\omega_0\|_{H_g ^\gamma} + \|\omega^\det\|^2_{L^{4}(0, T; H_g^{\gamma/2 + 1/4})}.
\end{align*}
Now we interpolate again $H_g^{\gamma /2+1/4}= [H^{-1}, H_g^{\gamma + 1/2}]_{1-\beta}$ with $\beta = (2\gamma + 1)/(4\gamma +6) \in (0,1)$ to get 
\begin{align}\label{2p-1p_interpolated}
    \|\omega^\det\|_{L^{4}(0, T; H_g^{\gamma + 1/2})} \lesssim \|\omega_0\|_{H_g^{\gamma}} + \|\omega^\det\|^{2\beta}_{L^{4}(0, T; H^{-1})}\|\omega^\det\|^{2-2\beta}_{L^{4}(0, T;H_g^{\gamma + 1/2})}.
\end{align}
Another round of interpolation also implies
\begin{align*}
    \|\omega^\det\|_{L^{4}(0, T; H^{-1})} &\lesssim  \|\omega^\det\|_{L^{4}(0, T; H^{-1/2})} \\
    &\lesssim \nu^{-1/4}( \nu^{1/2}\|\omega^\det\|_{L^{2}(0, T; L^2)})^{1/2}\|\omega^\det\|^{1/2}_{L^\infty(0, T; H^{-1})}\\
    &\lesssim  \nu^{-1/4} \|\omega_0\|_{H^{-1}},
\end{align*}
where we have used relation \eqref{det_nonlinear_en_estimate} in the last inequality. 
Plugging this estimate in \eqref{2p-1p_interpolated} we get, for $\delta =\beta/[2(1-\beta)]$
\begin{align*}
    \|\omega^\det\|_{L^{4}(0, T; H_g^{\gamma +1/2})} &\lesssim \|\omega_0\|_{H^{\gamma}_g} + (\nu^{-1/4} \|\omega_0\|_{H^{-1}})^{2\beta}\|\omega^\det\|^{2-2\beta}_{L^{4}(0, T;H_g^{\gamma + 1/2})} \\
    &\lesssim \|\omega_0\|_{H^{\gamma}_g} + \|\omega_0\|_{H^{-1}}^2 +\nu^{-\delta} \|\omega^\det\|^2_{L^{4}(0, T;H_g^{\gamma + 1/2})}.
\end{align*}
Since $\|\omega_0\|_{H^{\gamma}_g}\leq M$, we have shown the existence of $R=R(M,\gamma)\geq 0$, $\delta=\delta(\gamma) >0$, for which 
\begin{align*}
    \|\omega^\det\|_{L^{4}(0, T; H_g^{\gamma + 1/2})} \leq R \|\omega_0\|_{H^\gamma_g} + R\nu^{-\delta}\|\omega^\det\|^{2}_{L^{4}(0, T;H_g^{\gamma + 1/2})}.
\end{align*}
Now a classical bootstrapping argument (see \cite[Step 4-5, proof of Theorem 5.1]{agresti_global_2024}) allows us to deduce that if $\nu \geq (8R^2M)^{1/\delta}$, then 
\begin{align*}\|\omega^\det\|_{L^{4}(0, T; H_g^{\gamma + 1/2})} \leq 2R\|\omega_0\|_{H^\gamma_g}.\end{align*}
Finally, again using \eqref{eq:e.e._gamma} we get
\begin{align*}
     \sup_{s\in [0, T]}\|\omega^\det_s\|_{H_g^{\gamma}} + \|\omega^\det\|_{L^2(0, T; H^{\gamma +1}_g)} &\lesssim \|\omega_0\|_{H^\gamma_g} +  \|\omega^\det\|^2_{L^{4}(0, T; H^{\gamma/2+1/4}_g)} \\
    &\lesssim \|\omega_0\|_{H^\gamma_g} +  \|\omega^\det\|^2_{L^{4}(0, T; H^{\gamma+1/2}_g)}\\
    &\lesssim \|\omega_0\|_{H_g^\gamma} + 2R\|\omega_0\|^2_{H_g^\gamma}.
\end{align*}
Since $\|\omega_0 \|_{H_g^\gamma }\leq M$, we have obtained \eqref{det_goal}, with $K_0$ depending only on $M$ and $\gamma$. This concludes the proof. 
\end{proof}

\begin{remark}
\label{rmk:uniqueness}
    Let $p>4$ and $-\frac{1}{2}<r<-\frac{2}{p}$ such that $r\leq r_0<\gamma$ and le $R$ be larger than the right-hand side (including the hidden constant) appearing in \eqref{det_energy_estimates_2}. If $\omega_0$ additionally belongs to $B^{-2/p}_{g,2,p}$, then the solution given by \autoref{lem:dissipation} coincides with the unique $p$-solution of \eqref{main_det_cutoff}. By the assumption on $R,\ r,r_0$ if $\omega^{\det}$ is the solution given by previous lemma we have that 
    \begin{align*}
       \|\omega_t^{\det}\|_{H^r_g}\leq R. 
    \end{align*}
     Therefore it is enough to check that $\omega^{\det}$ belongs to the proper regularity class, in particular 
     \begin{align*}
         \omega^{\det}\in W^{1,p}(0,T;H^{-2}_g)\cap L^p(0,T;H^{0}_g)\hookrightarrow C([0,T];B^{-2/p}_{g,2,p}).
     \end{align*}
     Recall that from equation \eqref{ineq:nonlinear_cutoff_interp} there exists $k\in (0,1)$ such that
     \begin{align*}
          \|\LL_u \omega \|_{H^{-2}_g} \lesssim \|\omega\|^{1-k}_{H^0_g}\|\omega \|^{1+k}_{H^r_g}.
     \end{align*}
    Since $\gamma>-\frac{1}{2},$ we already know that $\omega^{\det}\in L^2(0,T;H^{0}_{g})$ and it is enough to check for $p_j=p\wedge(\frac{2}{(1-k)^j})$
    \begin{align*}
        \omega^{\det}\in L^{p_j}(0,T;H^{0}_{g})\implies  \omega^{\det}\in L^{p_{j+1}}(0,T;H^{0}_{g})\cap  W^{1,p_{j+1}}(0,T;H^{-2}_{g}).
    \end{align*}
    In order to do so, we apply maximal $L^p$ regularity for the operator $\nu\Delta+A$ on $H^{-2}_g$ and forcing $ f=-\phi_{R}(\|\omega^{R,\det}\|_{H^r_g})\LL_u \omega^{R,\det}.$ The maximal regularity for the operator $\nu\Delta+A$ can be proved as discussed in \autoref{app_technical_res}. Assume now that $\omega^{\det}\in L^{p_j}(0,T;H^{0}_{g})$, then by \eqref{ineq:nonlinear_cutoff_interp} we have
    $
      \|\LL_u \omega \|_{H^{-2}_g}\lesssim_R  \|\omega\|^{1-k}_{H^0_g}.
    $
    Therefore
    $
        f\in L^{p_{j+1}}(0,T;H^{-2}_{g})
    $
    and the claim follows.
\end{remark}

We are now ready to prove \autoref{thm:main}. 

\begin{proof}[Proof of \autoref{thm:main}]
Fix $r,r_0,\gamma$ satisfying
\begin{align*}
    -\frac{1}{2}<r<r_0<\gamma<-\frac{2}{p}.
\end{align*}
Since $B^{-2/p}_{g,2,p}\hookrightarrow H_g^\gamma$, \autoref{lem:dissipation} allows us to choose $\nu>0$ such that \eqref{main_det} admits a global solution $\omega^\det$ for every $\omega_0\in\B_p(M)$. Using also $H_g^{r_0}\hookrightarrow H_g^r$ and $H_g^{r_0}\hookrightarrow B^{-1/2}_{g,2,4},$ we can choose $R_0\geq1$ and $\eta_0>0$, uniformly over $\omega_0\in\B_p(M)$, such that
\begin{align}\label{eq:deterministic_bounds}
\sup_{t\geq0}\|\omega_t^\det\|_{H_g^r}\leq R_0,
\qquad
\sup_{t\geq0}e^{\eta_0t}
\|\omega_t^\det\|_{B^{-1/2}_{g,2,4}}\leq R_0.
\end{align}
With this choice of $\nu$, let $\eta\in(0,1)$ and $n_0$ be as in \autoref{prop:small_implies_global}, applied with $p=4$ and with $\eps$ replaced by $\eps/2$.
Set also $R=R_0+1$ and choose $T_0>0$ such that $R_0e^{-\eta_0T_0}\leq\eta/2.$
In particular, it holds $\|\omega_{T_0}^\det\|_{B^{-1/2}_{g,2,4}}\leq\eta/2.$
By \autoref{rmk:uniqueness}, $\omega^\det$ coincides with the p-solution of the deterministic cutoff equation \eqref{main_det_cutoff} with cutoff parameter $R$.

Let $\omega^n$ be the global solution of \eqref{main_ito_cutoff} with this cutoff.
By \autoref{thm: main scaling}, \eqref{eq:deterministic_bounds} and the embeddings above, denoting by
\begin{align*}
    \A_n\coloneq
\Big\{
\sup_{t\in[0,T_0]}\|\omega_t^n\|_{H_g^r}<R,\quad
\|\omega_{T_0}^n\|_{B^{-1/2}_{g,2,4}}\leq\eta
\Big\},
\end{align*}
it holds, for all sufficiently large $n\geq n_0$,
\begin{align*}
    \PP(\A_n^c)<\eps/2,
\end{align*}
uniformly over $\omega_0\in\B_p(M)$. Fix such an $n$. On $\A_n$, the cutoff remains inactive
up to time $T_0$. Therefore, by uniqueness and continuation, the maximal $p$-solution $(\omega,\tau)$ of \eqref{main_ito} satisfies $\tau>T_0$ and $\omega=\omega^n$ on $[0,T_0]$ on this event. Since $\A_n\in\F_{T_0}$, we can restart equation \eqref{main_ito} at time $T_0$ from the initial datum $\xi\coloneq\one_{\A_n}\omega_{T_0}^n.$
By construction,
\begin{equation}\label{eq:small_implies_global_proof}
\EE\big[\|\xi\|_{B^{-1/2}_{g,2,4}}^4\big]
\leq\eta^4\leq\eta.
\end{equation}
Thus, \autoref{prop:small_implies_global} implies that
the maximal $4$-solution $(\tilde{\omega},\zeta)$ starting from $\xi$
at time $T_0$, with the same noise coefficients and
Brownian motions, satisfies
\begin{align*}
    \PP(\zeta=\infty)>1-\eps/2.
\end{align*}
On $\A_n$, uniqueness implies $\tilde{\omega}=\omega$ after $T_0$ for as long as both solutions exist. On $\A_n\cap\{\zeta=\infty\}$, instantaneous regularization of $\tilde{\omega}$ and the continuation criterion for the maximal $p$-solution imply $\tau=\infty$. Therefore,
\begin{align*}
    \PP(\tau=\infty)
&\geq\PP\left(\A_n\cap\{\zeta=\infty\}\right)\\
&\geq1-\PP(\A_n^c)-\PP(\zeta<\infty)
>1-\eps.
\end{align*}
Finally, the claimed regularity follows from \autoref{prop:local_well_nonlinear}.
\end{proof}

\section{Enhanced dissipation -- Proof of \autoref{thm_global}}\label{sec:proof_global}
In this section, we prove the enhanced dissipation of velocity as stated in \autoref{thm_global}. In light of the proof of \autoref{thm:main}, it suffices to prove a global enhanced dissipation result for the 3D Navier-Stokes equations with a modified cutoff that accounts for the bounds already obtained via the scaling limit (see \cite[Section 7]{agresti_kniely_tang_2026} for a similar situation). We stress that the strategy used in the literature in the case of pure transport noise (see, e.g., \cite{agresti_kniely_tang_2026,flandoli_quantitative_2024,luo_enhanced_2021}) relies on the pathwise energy conservation, which is not available for the Lie-transport noise; see \autoref{r:lack_energy}. To prove enhanced dissipation for the cutoff problem, we instead exploit the hyperviscosity and the special structure of the noise; see \autoref{hyp_noise}.

This section is organized as follows. In Subsection \ref{ss:enhnced_cutoff}, we introduce another Navier-Stokes equation with cutoff and state a corresponding enhanced result. The proof of the latter is postponed to Subsection \ref{ss:enhanced_dissipation_cutoff}. The proof of \autoref{thm_global} is given in Subsection \ref{ss:enhanced_dissipation_proof}, deducing it from the corresponding result with cutoff. 

\subsection{Enhanced dissipation with Lie-transport noise and cutoff}
\label{ss:enhnced_cutoff}
Fix $p>4$ and consider the following 3D Navier-Stokes equations with (modified) cutoff:
\begin{equation}
\label{eq:ED_auxiliary_dynamics}
\begin{cases}
 \dd \omega+\left(\psi_R (t,\omega)\LL_{u}\omega-A\omega-\nu\Delta \omega\right)\,\dd t
       +\sqrt{\nu}\displaystyle\sum_{k\in I}
          \LL_{\sigma_k^n}\omega\,\dd W_t^k=0,\\
          u=K[\omega],\qquad 
 \omega(0)=\omega_0,
 \end{cases}
\end{equation}
where $R\geq 1$ and 
\begin{equation}
\label{eq:definition_modified_cutoff_psi}
\psi_R(t,\omega)\coloneq \psi \Big(R^{-1}\Big[\sup_{s\in[0,t]}\|\omega_s\|_{B^{-2/p}_{g,2,p}}^p
       +\int_0^t
          \big(\|\omega_s\|_{H_g^0}^p+
               \|\omega_s\|_{H_g^0}^4\big)\,\dd s\Big] \Big).
\end{equation}
where $\psi:\R\to [0,1]$ is a smooth, non-increasing function such that $\psi|_{[0,1]}=1$ and $\psi|_{[2,\infty]}=0$. Unique, maximal $p$-solutions to the above system can be defined as in \autoref{def:local_well_posed}. Similar to \autoref{prop:global_well_cutoff}, the above admits a unique global $p$-solution $\omega$ (i.e., has lifetime $\tau=\infty$ a.s.).

\smallskip

The key result proven in this section reads as follows.

\begin{proposition}[Enhanced dissipation for the truncated equation]
\label{prop:ED_truncated}
Let \autoref{hp_correction} be satisfied.
Let $p>4$, $\chi,b\geq1$, $R>1$ and $\nu\geq \chi/\pi^2$.
There exists $n_0=n_0(\nu,b,g)$ such that, for all $n\geq n_0$, every $\omega_0\in B^{-2/p}_{g,2,p}$, the unique global $p$-solution $\omega$ to the 3D Navier-Stokes equations with modified cutoff \eqref{eq:ED_auxiliary_dynamics} satisfies
\begin{equation}\label{eq:ED_truncated_decay}
 \|\omega_t\|_{H^{-1}}
 \leq D e^{-\chi t}\|\omega_0\|_{H^{-1}} \ \text{ a.s. for all }t>0,
\end{equation}
where $D$ is a random constant satisfying $ \EE[D^b]\lesssim_{R,b,g}1$.
\end{proposition}

The proof of the above is given in \autoref{ss:enhanced_dissipation_cutoff}. In the next subsection, we deduce \autoref{thm_global} from \autoref{prop:ED_truncated}.

\subsection{Proof of \autoref{thm_global}}
\label{ss:enhanced_dissipation_proof}

\begin{proof}[Proof of \autoref{thm_global}]
Let the assumptions of \autoref{thm_global} be satisfied. 
Let $(\omega,\tau)$ be the unique maximal $p$-solution to \eqref{main_ito}. 
From \autoref{prop:ED_truncated}, it suffices to show the existence of a constant $R=R(p,\eps,M,\chi,b)>1$ such that, for all sufficiently large $n$, 
\begin{equation}
\label{eq:proof_enhnaced_dissipation_reduction}
\PP(\tau=\infty,\, \E_\omega(\infty)\mathrel < R)>1-\varepsilon
\end{equation}
where 
\begin{align*}
\E_\omega(t)\coloneq\sup_{s\in[0,t]}\|\omega_s\|_{B^{-2/p}_{g,2,p}}^p+\int_0^t \left(\|\omega_s\|_{H_g^0}^p+ \|\omega_s\|_{H_g^0}^4\right)\,\dd s.
\end{align*}
Indeed, if \eqref{eq:proof_enhnaced_dissipation_reduction} holds, then letting 
\begin{align*}
\mu\coloneq \inf\{t\in [0,\tau)\,:\,
\E_\omega(t)\geq R\},
\end{align*}
where $\inf\emptyset\coloneq \infty$, we have $\mu=\infty$ on $\{\tau=\infty,\ \E_\omega(\infty)<R\}$, so $\PP(\mu=\infty)>1-\eps$. On $\{\tau<\infty\}$, the blow-up criterion in \autoref{SMR_blowup+instreg} gives $\sup_{t<\tau}\E_\omega(t)=\infty$ a.s. Therefore, $\mu\leq\tau$ a.s., with $\mu<\tau$ on $\{\tau<\infty\}$. Since $\psi_R(t,\omega)=1$ for $t<\mu$, the couple $(\omega,\mu)$ is a local $p$-solution to \eqref{eq:ED_auxiliary_dynamics}. By the uniqueness for \eqref{eq:ED_auxiliary_dynamics} and \autoref{prop:ED_truncated}, the conclusion of \autoref{thm_global} follows.
Therefore, it suffices to prove \eqref{eq:proof_enhnaced_dissipation_reduction}.
To prove the latter, we partially go over again the proof of \autoref{thm:main} and derive stopped small-data estimates from \eqref{smr_halfline_inequality}.

Fix $-1/2<r<r_0<\gamma<-2/p$. Due to $B^{-2/p}_{g,2,p}\hookrightarrow H^\gamma_g$, we may choose $\nu\geq\chi/\pi^2$ so that \autoref{lem:dissipation} applies uniformly to $\omega_0\in\B_p(M)$. Next, recall that \autoref{lem:nonlinear_est} and \eqref{real_interpolation} give
\begin{equation}\label{eq:critical_trace_nonlinearity}
 \|\LL_{K[v]}v\|_{H^{-2}_g}\leq C\|v\|_{H^{-1/4}_g}^2\leq 
 C_0\|v\|_{B^{-1/2}_{g,2,4}}\|v\|_{H^0_g},
\end{equation}
for some $C_0>0$ independent of $v$.  For $q\in\{4,p\}$, let $C_q$ be the constant in \eqref{smr_halfline_inequality} with exponent $q$. By \autoref{SMR_halfline_prop}, $C_q$ is independent of $\omega_0$ and is uniform for $n$ sufficiently large. Let $\eta\in(0,1)$ be such that
\begin{align*}
 C_q(C_0\eta)^q\leq\tfrac12\quad\text{for }q\in\{4,p\},\qquad 2C_4\eta^4<\eps/4.
\end{align*}
With this choice of $\eta$, we now go over again the argument in the proof of \autoref{thm:main}. By \eqref{det_energy_estimates_2} and $H^{r_0}_g\hookrightarrow B^{-1/2}_{g,2,4}$, we can fix $R_0>1$ and subsequently $T_0>0$ such that
\begin{align*}
 \sup_{t\geq0}\|\omega_t^\det\|_{H^r_g}\leq R_0-1,
 \qquad \|\omega_{T_0}^\det\|_{B^{-1/2}_{g,2,4}}\leq\eta^2/2
\end{align*}
for every $\omega_0\in\B_p(M)$. Let $\omega^n$ be the solution to \eqref{main_ito_cutoff} with $R=R_0$. By \autoref{rmk:uniqueness}, the solution of the corresponding deterministic equation is $\omega^\det$. Thus \autoref{thm: main scaling}, applied in $C([0,T_0];H^{r_0}_g)$, gives
\begin{align*}
 \PP(\A_n^c)<\eps/4,\qquad
 \A_n\coloneq\Big\{\sup_{t\in[0,T_0]}\|\omega_t^n\|_{H^r_g}<R_0,\
\|\omega_{T_0}^n\|_{B^{-1/2}_{g,2,4}}\leq\eta^2\Big\},
\end{align*}
uniformly over $\omega_0\in \B_p(M)$ for all sufficiently large $n$. As in the proof of \autoref{thm:main}, on $\A_n$, the cutoff $\phi_{R_0}(\|\omega_t^n\|_{H^r_g})=1$ for all $t\leq T_0$. Hence, by maximality, $\tau>T_0$ and $\omega=\omega^n$ on $[0,T_0]\times \A_n$.

\smallskip

In the following, we quantify the small-implies-global argument as used below \eqref{eq:small_implies_global_proof}.
Let $(v,\zeta)$ be the unique maximal $p$-solution of \eqref{main_ito} starting at $T_0$ with initial datum $\xi=\one_{\A_n}\omega_{T_0}^n$ (note that $\A_n\in\F_{T_0}$ by definition). By construction and the estimate \eqref{a_priori_estimate_cutoff},
\begin{align*}
 \|\xi\|_{B^{-1/2}_{g,2,4}}\leq\eta^2\quad\text{a.s.},\qquad
 \EE\big[\|\xi\|_{B^{-2/p}_{g,2,p}}^p\big]\leq C(T_0)(1+M^p).
\end{align*}
In particular, $\EE\big[\|\xi\|_{B^{-1/2}_{g,2,4}}^4\big]\leq\eta^8$.
Next, let
\begin{align*}
 \sigma\coloneq\inf\{t\in[T_0,\zeta)\,:\,\|v_t\|_{B^{-1/2}_{g,2,4}}\geq\eta\}\quad\text{ where }\quad \inf\emptyset\coloneq \zeta.
\end{align*} 
For $T_0\leq t<\sigma$, \eqref{eq:critical_trace_nonlinearity} gives
\begin{align*}
 \|\LL_{K[v_t]}v_t\|_{H^{-2}_g}\leq C_0\eta\|v_t\|_{H^0_g}.
\end{align*}
Next, we apply \eqref{smr_halfline_inequality} with $q=4$ and $q=p$ on the stochastic interval $[T_0,\sigma]$. By the previously displayed estimate, the forcing term can be bounded by 
\begin{align*}
 C_q\EE\Big[\int_{T_0}^{\sigma}\|\LL_{K[v_t]}v_t\|_{H^{-2}_g}^q\,\dd t\Big]
 \leq C_q(C_0\eta)^q\EE\Big[\int_{T_0}^{\sigma}\|v_t\|_{H^0_g}^q\,\dd t\Big].
\end{align*} By the definition of $\eta,$ we may absorb this term into the left-hand side and obtain
\begin{equation}\label{eq:small_data_stopped_estimate}
 \EE\Big[\sup_{T_0\leq t<\sigma}\|v_t\|_{B^{-2/q}_{g,2,q}}^q\Big]
 +\EE\Big[\int_{T_0}^\sigma\|v_t\|_{H^0_g}^q\,\dd t\Big]
 \leq 2C_q\EE\big[\|\xi\|_{B^{-2/q}_{g,2,q}}^q\big],\qquad q\in\{4,p\}.
\end{equation}
Taking $q=p$ in \eqref{eq:small_data_stopped_estimate}, as $\EE\big[\|\xi\|_{B^{-2/p}_{g,2,p}}^p\big]<\infty$, the blow-up criterion  \autoref{SMR_blowup+instreg} readily gives $\PP(\sigma=\zeta<\infty)=0$. Moreover, by the definition of $\sigma$ and $v\in C([T_0,\zeta);B^{-2/p}_{g,2,p})$ a.s., it follows that 
\begin{align*}
 \sup_{T_0\leq t<\sigma}\|v_t\|_{B^{-1/2}_{g,2,4}}=\eta
 \qquad\text{a.s.\ on }\{\sigma<\infty\}
\end{align*} 
Using \eqref{eq:small_data_stopped_estimate} with $q=4$, we obtain
\begin{align*}
\mathbb P(\sigma<\infty)
\leq \frac{1}{\eta^4}\,
     \mathbb E\Big[\sup_{T_0\leq t<\sigma}\|v_t\|_{B^{-1/2}_{g,2,4}}^4\Big]
\leq \frac{2C_4}{\eta^4}\,
     \mathbb E\big[\|\xi\|_{B^{-1/2}_{g,2,4}}^4\big]
\leq 2C_4\eta^4.
\end{align*}
In particular, from the choice of $\eta$, we have $\PP(\sigma<\infty)\leq 2C_4\eta^4<\eps/4$. In addition, since $\PP(\A_n^c)<\eps/4$, the event $\G_n=\A_n\cap\{\sigma=\infty\}$ satisfies $\PP(\G_n)> 1-\eps/2$. By uniqueness and maximality, on $\G_n$ we have $\tau=\infty$ a.s., and $\omega=v$ a.e.\ on $[T_0,\infty)$.

It remains to prove \eqref{eq:proof_enhnaced_dissipation_reduction}. Recall that, on $\G_n$, we have $\omega=\omega^n$ on $[0,T_0]$ and $\omega=v$ on $[T_0,\infty)$, with $\sigma=\infty$. 
Therefore,
\begin{align*}
 \EE\left[\one_{\G_n}\E_\omega(\infty)\right]
 &\leq \EE\Big[\one_{\G_n}\Big(
     \sup_{0\leq t\leq T_0}\|\omega_t^n\|_{B^{-2/p}_{g,2,p}}^p
     +\int_0^{T_0}\big(\|\omega_t^n\|_{H_g^0}^p+\|\omega_t^n\|_{H_g^0}^4\big)\,\dd t
     \Big)\Big]\\
 &+\EE\Big[\one_{\G_n}\Big(
     \sup_{T_0\leq t<\sigma}\|v_t\|_{B^{-2/p}_{g,2,p}}^p
     +\int_{T_0}^\sigma\|v_t\|_{H_g^0}^p\,\dd t
     \Big)\Big]\\
 &+\EE\Big[\one_{\G_n}\int_{T_0}^\sigma\|v_t\|_{H_g^0}^4\,\dd t\Big].
\end{align*}
Combining \eqref{a_priori_estimate_cutoff}, and \eqref{eq:small_data_stopped_estimate} with $q\in \{4,p\}$, it follows that 
\begin{align*}
 \EE[\one_{\G_n}\E_\omega(\infty)]
 \leq \EE\big[\E_{\omega^n}(T_0)\big]
      +2C_p\EE\big[\|\xi\|_{B^{-2/p}_{g,2,p}}^p\big]
      +2C_4\EE\big[\|\xi\|_{B^{-1/2}_{g,2,4}}^4\big]\leq K_0,
\end{align*}
for some constant $K_0$ independent of $\omega_0\in\B_p(M)$, and $n$ sufficiently large.

Choose $R>\max\{M^p,2K_0/\eps\}$. From the Markov inequality, it follows that 
\begin{align*}
 \PP\left(\G_n\cap\{\E_\omega(\infty)\geq R\}\right)
 =\PP\left(\one_{\G_n}\E_\omega(\infty)\geq R\right) \leq \frac{K_0}{R}<\frac{\eps}{2}.
\end{align*}
Since $\tau=\infty$ on $\G_n$ and $\PP(\G_n)>1-\eps/2$, it follows that
\begin{align*}
 \PP(\tau=\infty,\ \E_\omega(\infty)<R)
 &\geq \PP(\G_n\cap\{\E_\omega(\infty)<R\})\\
 &=\PP(\G_n)-\PP(\G_n\cap\{\E_\omega(\infty)\geq R\})
 >1-\varepsilon.
\end{align*}
This proves \eqref{eq:proof_enhnaced_dissipation_reduction}. Thus \autoref{thm_global} follows from \autoref{prop:ED_truncated}.
\end{proof}

\subsection{Proof of \autoref{prop:ED_truncated}}
\label{ss:enhanced_dissipation_cutoff}
In this subsection, we prove \autoref{prop:ED_truncated}. As commented at the beginning of this section, due to the lack of pathwise energy inequalities, the known arguments in the context of scaling limits (see e.g., \cite{flandoli_quantitative_2024,luo_enhanced_2021}) cannot be applied. Instead, we exploit the Fourier localization of the noise coefficients $(\sigma^n_k)_k$ and a modified energy to capture the enhanced dissipation effect of the Lie-transport noise.

\begin{proof}[Proof of \autoref{prop:ED_truncated}]
Let $\omega$ be the unique global $p$-solution to \eqref{eq:ED_auxiliary_dynamics}. Without loss of generality, we may assume $\omega_0\neq 0$. 
As above, for $2\pi\leq\Lambda<n/2$, let $P_{\leq \Lambda}$ denote the Fourier projection onto the modes $0<|\lambda|\leq\Lambda$. Similarly, $P_{>\Lambda}\coloneq\id-P_{\leq \Lambda}$. For $\alpha\in(0,1)$, set
\begin{equation}\label{eq:ED_energy_definition}
 T_\Lambda=P_{\leq \Lambda}+\alpha P_{>\Lambda},\qquad
 \E_\Lambda(\omega)=\|P_{\leq \Lambda} \omega\|_{H^{-1}}^2+\alpha\|P_{>\Lambda} \omega\|_{H^{-1}}^2.
\end{equation}
In particular,
\begin{equation}\label{eq:ED_energy_equivalence}
 \alpha\|\omega\|_{H^{-1}}^2\leq\E_\Lambda(\omega)
 =\langle \omega,T_\Lambda \omega\rangle_{H^{-1}}\leq\|\omega\|_{H^{-1}}^2.
\end{equation}
We preliminarily compute the evolution of $\E_\Lambda(\omega_t)$ by It\^o's formula and provide suitable bounds on the terms appearing in its evolution. These will subsequently be used in the next two steps to prove \eqref{eq:ED_truncated_decay}. Applying It\^o's formula to $\E_\Lambda(\omega_t)$ as in \cite[Theorem 4.2.5]{LR15}, we obtain
\begin{align}
\label{eq:Ito_identity_modified_energy}
 \dd\E_\Lambda(\omega)-2\langle(A+\nu\Delta)\omega,T_\Lambda \omega\rangle_{H^{-1}}\,\dd t
 &=-2\psi_R(t,\omega)\langle\LL_u \omega,(-\Delta)^{-1}T_\Lambda \omega\rangle\,\dd t\\
 \nonumber
 &\quad+2\nu\sum_{k\in I}\E_\Lambda(\LL_{\sigma_k^n}\omega)\,\dd t+\dd M_t,
\end{align}
where
\begin{align*}
 M_t=-2\sqrt{\nu}\sum_{k\in I}\int_0^t
 \langle\LL_{\sigma_k^n}\omega_s,T_\Lambda \omega_s\rangle_{H^{-1}}\,\dd W_s^k
\end{align*}
is a real continuous local martingale. Since $\Delta$ acts on Fourier as a multiplier, it holds
\begin{align*}
 2\nu\langle\Delta \omega,T_\Lambda \omega\rangle_{H^{-1}}
 =-2\nu\|P_{\leq\Lambda}\omega\|_{L^2}^2-2\nu\alpha\|P_{>\Lambda}\omega\|_{L^2}^2.
\end{align*}
Next, we estimate the It\^o correction. For $k=(h,j)$ with $\theta_h^n\neq0$, 
\begin{align*}
( \LL_{\sigma_k^n} P_{\leq\Lambda}v, \LL_{\sigma_k^n} P_{>\Lambda}v)_{H^{-1}}=0\quad \text{ for }v\in L^2.
\end{align*}
Indeed, on the Fourier side, the operator $\LL_{\sigma_k^n}$ acts as a shift since $\sigma_k^n=\theta_h^n a_{h,j}e^{ih\cdot x}$ for $k=(h,j)$. For frequencies $|\lambda|\leq\Lambda$, we have $|\lambda+h|\geq n-\Lambda>\Lambda$ as $\Lambda<n/2$. In particular, for $v\in L^2$,
\begin{align*}
 P_{\leq\Lambda} \LL_{\sigma_k^n} P_{\leq\Lambda}v=0,\qquad\text{ and }\quad 
 P_{>\Lambda} \LL_{\sigma_k^n} P_{\leq\Lambda}v=\LL_{\sigma_k^n} P_{\leq\Lambda}v.
\end{align*}
Therefore, expanding $\E_\Lambda( \LL_{\sigma_k^n} \omega)$ with $\omega=P_{\leq\Lambda}\omega+P_{>\Lambda}\omega$, we obtain
\begin{align*}
 \E_\Lambda( \LL_{\sigma_k^n} \omega)
 &=\|P_{\leq\Lambda}\LL_{\sigma_k^n} \omega\|_{H^{-1}}^2+\alpha\|P_{>\Lambda} \LL_{\sigma_k^n} \omega\|_{H^{-1}}^2\\
 &=\|P_{\leq\Lambda}\LL_{\sigma_k^n} P_{>\Lambda}\omega\|_{H^{-1}}^2
   +\alpha\|\LL_{\sigma_k^n} P_{\leq\Lambda}\omega\|_{H^{-1}}^2
   +\alpha\|P_{>\Lambda}\LL_{\sigma_k^n} P_{>\Lambda}\omega\|_{H^{-1}}^2\\
 &=\alpha\|\LL_{\sigma_k^n} P_{\leq\Lambda}\omega\|_{H^{-1}}^2+\E_\Lambda(\LL_{\sigma_k^n} P_{>\Lambda}\omega).
\end{align*}
Using the upper bound in \eqref{eq:ED_energy_equivalence} and \autoref{lem:noise_estimate}, the above yields
\begin{align}\label{eq:ED_weighted_noise}
 2\nu\sum_{k\in I}\E_\Lambda(\LL_{\sigma_k^n}\omega)
 &\leq 2\nu\alpha\sum_{k\in I}\|\LL_{\sigma_k^n}P_{\leq\Lambda}\omega\|_{H^{-1}}^2
       +2\nu\sum_{k\in I}\|\LL_{\sigma_k^n}P_{>\Lambda}\omega\|_{H^{-1}}^2\notag\\
 &\leq 6\nu\alpha\|P_{\leq\Lambda}\omega\|_{L^2}^2+6\nu\|P_{>\Lambda}\omega\|_{L^2}^2.
\end{align}
Finally, we estimate the nonlinearity. Integrating by parts and using $\LL_u \omega = \nabla \times (\nabla \cdot (u\otimes u))$, we have
\begin{align*}
|\langle\LL_u \omega,(-\Delta)^{-1}T_\Lambda \omega\rangle|&=\left|\int_{\T^3}(u\otimes u):\nabla T_\Lambda u\,\dd x\right|\\
 &\leq\|u\|_{L^6}\|u\|_{L^3}\|\nabla T_\Lambda u\|_{L^2}
 \lesssim\|u\|_{L^2}^{1/2}\|\nabla u\|_{L^2}^{5/2}.
\end{align*}
Since $\|u\|_{L^2}\eqsim \|\omega\|_{H^{-1}}$ and $\|\nabla u\|_{L^2}=\|\omega\|_{L^2}\lesssim_g\|\omega\|_{H_g^0}$, from Young's inequality and the lower bound in \eqref{eq:ED_energy_equivalence}, we obtain 
\begin{equation}\label{eq:ED_nonlinear_energy}
 \psi_R(t,\omega)\left|\langle\LL_u \omega,(-\Delta)^{-1}T_\Lambda \omega\rangle\right|
 \leq \frac{\alpha}{4}\|\omega\|_{H_g^0}^2
       +\frac{C_{g}}{\alpha^4}\psi_R(t,\omega)\|\omega\|_{H_g^0}^4\E_\Lambda(\omega),
\end{equation}
where we used also $\psi_R(t,\omega)\leq1$ for the first term. Recall that the constants $p>4$, $\chi,b\geq1$, $R>1$ and $\nu\geq \chi/\pi^2$ are fixed as in \autoref{prop:ED_truncated}.

\smallskip

\emph{Step 1.} In this step we prove the following claim: set $Y_t:=\E_\Lambda(\omega_t)^{m/2}$, where $m$ is the smallest integer $>2\vee b$, and $\alpha=\frac{1}{6(m-1)}$. Let $n_0$ and $2\pi \leq \Lambda<n_0/2$ be such that
\begin{equation}\label{eq:ED_frequency_choice}
 g(\Lambda)\geq \frac{6}{\alpha}\nu(m-1).
\end{equation}
Then, for all $n\geq n_0$,
\begin{equation}\label{eq:ED_weighted_Ito}
 \dd Y_t+m\chi Y_t\,\dd t
 \leq c_{b,g}\psi_R(t,\omega)\|\omega_t\|_{H_g^0}^4Y_t\,\dd t+\dd\M_t,
\end{equation}
where $\M_t$ is a continuous local martingale starting at $0$. Note that the choice in \eqref{eq:ED_frequency_choice} is possible as $g(\Lambda)\to\infty$ (e.g., choose such a $\Lambda$, then take any integer $n_0>2\Lambda$).

Since it holds 
\begin{align*}
\mathrm d[M]_t&=8\nu\sum_{k\in I}
\left|\left\langle
\LL_{\sigma_k^n}\omega_t,T_\Lambda\omega_t
\right\rangle_{H^{-1}}\right|^2 \mathrm dt\leq8\nu \E_\Lambda(\omega_t)
\sum_{k\in I}\E_\Lambda(\LL_{\sigma_k^n}\omega_t)\mathrm dt,
\end{align*}
by It\^o's formula and \eqref{eq:Ito_identity_modified_energy}, as well as using the estimates \eqref{eq:ED_weighted_noise} and \eqref{eq:ED_nonlinear_energy}, we obtain
\begin{align}
\label{eq:estimate_Yt_proof_enhanced}
 \dd Y_t
 &\leq m\E_\Lambda(\omega_t)^{m/2-1}\left[
      -\|P_{\leq\Lambda}\omega_t\|_{H_g^0}^2-\alpha\|P_{>\Lambda}\omega_t\|_{H_g^0}^2\right.\\
      \nonumber
 & \left.-\nu\|P_{\leq\Lambda}\omega_t\|_{L^2}^2-\nu\alpha\|P_{>\Lambda}\omega_t\|_{L^2}^2\right.\\
 \nonumber
 & \left.+3\nu(m-1)\alpha\|P_{\leq\Lambda}\omega_t\|_{L^2}^2
          +3\nu(m-1)\|P_{>\Lambda}\omega_t\|_{L^2}^2
          +\frac{\alpha}{4}\|\omega_t\|_{H_g^0}^2\right]\,\dd t\\
          \nonumber
 & +C_{b,g}\psi_R(t,\omega)\|\omega_t\|_{H_g^0}^4Y_t\,\dd t+\dd\M_t \\
 \nonumber
& =m\E_\Lambda(\omega_t)^{m/2-1}\left[
      -\|P_{\leq\Lambda}\omega_t\|_{H_g^0}^2-\alpha\|P_{>\Lambda}\omega_t\|_{H_g^0}^2+\frac{\alpha}{4}\|\omega_t\|_{H_g^0}^2\right.\\
      \nonumber
 & \left.-(1-3(m-1)\alpha)\nu\|P_{\leq\Lambda}\omega_t\|_{L^2}^2-\nu(\alpha-3(m-1))\|P_{>\Lambda}\omega_t\|_{L^2}^2
          \right]\,\dd t\\
          \nonumber
 & +C_{b,g}\psi_R(t,\omega)\|\omega_t\|_{H_g^0}^4Y_t\,\dd t+\dd\M_t  
\end{align}
where $\M$ is a continuous local martingale.

By the choice of $\alpha$, we have $1-3(m-1)\alpha=1/2$. To estimate the last term in the bracket, recall that $g(|\lambda|)\geq g(\Lambda)$ for $|\lambda|>\Lambda$. Hence, by \eqref{eq:ED_frequency_choice},
\begin{align*}
 3\nu(m-1)\|P_{>\Lambda}\omega_t\|_{L^2}^2
 \leq \frac{\alpha}{2}g(\Lambda)\|P_{>\Lambda}\omega_t\|_{L^2}^2
 \leq \frac{\alpha}{2}\|P_{>\Lambda}\omega_t\|_{H_g^0}^2.
\end{align*}
Therefore, we may absorb the above term into the dissipation due to the hyperdissipative operator $A$. The dissipation due to $A$ gives the contribution
\begin{align*}
 \|P_{\leq\Lambda}\omega_t\|_{H_g^0}^2+\frac{\alpha}{2}\|P_{>\Lambda}\omega_t\|_{H_g^0}^2
 \geq\frac{\alpha}{2}\|\omega_t\|_{H_g^0}^2.
\end{align*}
The above can also be used to absorb the term $\frac{\alpha}{4}\|\omega_t\|_{H_g^0}^2$. Collecting the previous estimates, it follows that the term in the square bracket on the right-hand side of \eqref{eq:estimate_Yt_proof_enhanced} can be bounded from above by
\begin{align*}
 -\frac{\nu}{2}\|P_{\leq\Lambda}\omega_t\|_{L^2}^2-\nu\alpha\|P_{>\Lambda}\omega_t\|_{L^2}^2
 \leq -\frac{\nu(2\pi)^2}{2}\E_\Lambda(\omega_t)\leq-\chi\E_\Lambda(\omega_t),
\end{align*}
where we used Poincar\'e's inequality and $\nu\geq\chi/\pi^2$. This completes the proof of \eqref{eq:ED_weighted_Ito}.

\smallskip

\emph{Step 2. Conclusion.} First, we estimate the drift in \eqref{eq:ED_weighted_Ito}, exploiting the presence of the cutoff. 
From \eqref{eq:definition_modified_cutoff_psi}, it follows that 
\begin{equation}\label{eq:ED_compensator_budget}
 \int_0^t\psi_R(s,\omega)\|\omega_s\|_{H_g^0}^4\,\dd s
 \leq \int_0^{t\wedge\mu_R}\|\omega_s\|_{H_g^0}^4\,\dd s\leq 2R,
\end{equation}
where $\mu_R \coloneq \inf\{t\in [0,\infty)\,:\, \int_0^t\|\omega_s\|_{H^0_g}^4\,\dd s \geq 2R\}$ is the usual exit time (as above, $\inf\emptyset\coloneq \infty$). Thus, the time integral of the drift coefficient on the right-hand side of \eqref{eq:ED_weighted_Ito} is uniformly bounded. Next, consider
\begin{equation}\label{eq:ED_supermartingale}
 S_t=\exp\Big(m\chi t-c_{b,g}\int_0^t
             \psi_R(s,\omega)\|\omega_s\|_{H_g^0}^4\,\dd s\Big) Y_t.
\end{equation}
By \eqref{eq:ED_weighted_Ito} and It\^o's formula, $S$ is a nonnegative local supermartingale. Moreover, as $S_0=Y_0 = \E_\Lambda (\omega_0)^{m/2}<\infty$ is deterministic, $S$ is a supermartingale. By Doob's $(1,1)$-maximal inequality, it follows that  
\begin{equation}\label{eq:ED_maximal_tail}
  \PP\Big(\sup_{t>0}S_t >r\Big)\leq r^{-1}\E_\Lambda (\omega_0)^{m/2}
 \quad\text{for all }r>0.
\end{equation}
Therefore, by \eqref{eq:ED_compensator_budget},
\begin{align*}
\E_\Lambda (\omega_t)^{m/2}
&=Y_t = S_t \exp\Big(-m\chi t+c_{b,g}\int_0^t
             \psi_R(s,\omega)\|\omega_s\|_{H_g^0}^4\,\dd s\Big)\\
&\leq   S_t e^{-m\chi t} e^{C_{b,g}R}.
\end{align*}
Hence, from \eqref{eq:ED_energy_equivalence} and the choice of $\alpha$, it follows that \eqref{eq:ED_truncated_decay} holds with 
\begin{align*}
D\coloneq \alpha^{-1/2}\Big[e^{C_{b,g}R}\left(\sup_{t>0}S_t\right) S_0^{-1}\Big]^{1/m}.
\end{align*}
Since $b/m<1$, integrating \eqref{eq:ED_maximal_tail} yields $\EE [D^b] \lesssim_{R,b,g} 1$.
This concludes the proof.
\end{proof}

\begin{appendix}
\section{Technical results}\label{app_technical_res}
\subsection{Proof of some Interpolation inequalities}\label{sec:interpolation}
\begin{proof}[Proof of \autoref{lemma_interpolation}]
  Relation \eqref{identification_p2} is easy since the $\rho_j$ have almost disjoint support. Relation \eqref{real_interpolation} is well-known to experts, and follows almost verbatim from the proof of \cite[Theorem 14.3.1 and Theorem 14.4.31]{Analysis3}.
  We conclude the proof by providing some details about the explicit constants in the interpolation estimates. 
  
  \noindent\textit{Step 1. (Equivalence of norms)} There exists a constant $C=C(r,g)$ for which 
  \begin{align*}C^{-1}\|f\|_{B_{g,2,2}^r}\leq \|f\|_{H_{g}^r}\leq C\|f\|_{B_{g,2,2}^r}\end{align*}
  We show the result for a scalar $f$. The extension to a vector-valued function is trivial. 
  Recall that the partition of unity $(\chi, \rho)$ is made such that $\rho$ has support on $[1/2, 2]$. Obviously, for $j\geq 0$, $\rho_j(k)\neq 0 $ implies $2^{j-1}\leq |k|\leq 2^{j+1}$. Moreover, by the Plancherel identity 
  $\|\Delta_j f\|_{L^2}^2= \sum_{k\in 2\pi\Z_0^3} \rho_j^2(k)|\widehat{f}(k)|^2$. Finally, there exists a constant $C$ such that $ 1 \leq C\sum_{j\geq -1} \rho^2_j(k) \leq C' $. 
  Thus, for every $r\in\mathbb R$, the doubling property \autoref{hp_correction2} gives
\begin{align*}
    \|f\|_{B^r_{g,2,2}}^2
    &= \sum_{j\geq-1}2^{2rj}g(2^j)^{r+1}
       \sum_{k\in 2\pi\Z^3_0} \rho_j(k)^2|\widehat{f}(k)|^2 \\
    &\eqsim_{r,g}
       \sum_{k\in 2\pi\Z^3_0} |k|^{2r}g(|k|)^{r+1}|\widehat{f}(k)|^2
       \sum_{j\geq-1}\rho_j(k)^2 \\
    &\eqsim_{r,g}\|f\|_{H_g^r}^2.
\end{align*}
  \textit{Step 2. (Interpolation)}
  Let $s_\theta = \theta s_1 + (1-\theta) s_0$. We begin by showing 
  \begin{align}
      \|f\|_{B_{g,2,1}^{s_\theta}} \leq C_\theta \|f\|^{\theta}_{B_{g,2,\infty}^{s_1}} \|f\|^{1-\theta}_{B_{g,2,\infty}^{s_0}}
      \end{align}
      The above follows by reiteration of real interpolation (see e.g.\ \cite[Theorem 3.5.3]{MR482275}) and \eqref{real_interpolation}. For the reader's convenience, we provide some details. We have 
      \begin{align*}
           \|f\|_{B_{g,2,1}^{s_\theta}}& = \sum_{j\geq -1}2^{s_\theta j}g(2^j)^{\frac{s_\theta+1}{2}}\|\Delta_j f\|_{L^2} \\
           &= \sum_{j= -1}^N2^{s_\theta j}g(2^j)^{\frac{s_\theta+1}{2}}\|\Delta_j f\|_{L^2} + \sum_{j> N}2^{s_\theta j}g(2^j)^{\frac{s_\theta+1}{2}}\|\Delta_j f\|_{L^2} \\
           & \leq \|f\|_{B^{s_0}_{g,2,\infty}} \sum_{j= -1}^N2^{(s_\theta - s_0)j}g(2^j)^{\frac{s_\theta-s_0}{2}} + \|f\|_{B^{s_1}_{g,2,\infty}} \sum_{j> N}2^{(s_\theta - s_1)j}g(2^j)^{\frac{s_\theta-s_1}{2}}
      \end{align*}
      Now, since $g$ is increasing, by $s_0 < s_\theta < s_1$ we have
      \begin{align*}
          \sum_{j= -1}^N2^{(s_\theta - s_0)j}g(2^j)^{\frac{s_\theta-s_0}{2}} &
          \leq g(2^N)^{\frac{s_\theta -s_0}{2}}\sum_{j= -1}^N2^{(s_\theta - s_0)j} \leq C_{\theta,0} g(2^N)^{\frac{s_\theta -s_0}{2}}2^{N(s_\theta - s_0)} \\ \sum_{j> N}2^{(s_\theta - s_1)j}g(2^j)^{\frac{s_\theta-s_1}{2}}  &\leq g(2^N)^{\frac{s_\theta -s_1}{2}}\sum_{j> N}2^{(s_\theta - s_1)j}\leq C_{\theta,1} g(2^N)^{\frac{s_\theta -s_1}{2}}2^{N(s_\theta - s_1)}, 
        \end{align*}
        where the constant $C_{\theta, 0}$ and $C_{\theta, 1}$ can be chosen independent of $N$. Now we look for $N$ such that 
        \begin{align*}
            \|f\|_{B^{s_0}_{g,2,\infty}}g(2^N)^{\frac{s_\theta -s_0}{2}}2^{N(s_\theta - s_0)} \eqsim \|f\|_{B^{s_1}_{g,2,\infty}} g(2^N)^{\frac{s_\theta -s_1}{2}} 2^{N(s_\theta - s_1)},
        \end{align*}
        i.e., setting $\beta\coloneq  2^{N}\sqrt{g(2^N)}$ we choose $N$ (possibly depending on $f$) such that for some large but fixed constant $C$ (independent of $f$) it holds
        \begin{align*}
            C^{-1}\frac{\|f\|_{B^{s_1}_{g,2,\infty}}}{ \|f\|_{B^{s_0}_{g,2,\infty}}}\leq  \beta^{s_1 - s_0} \leq C\frac{\|f\|_{B^{s_1}_{g,2,\infty}}}{ \|f\|_{B^{s_0}_{g,2,\infty}}} 
        \end{align*}
        Noticing that $(s_\theta -s_0)/(s_1-s_0)=\theta$ we get
        \begin{align*}
              \|f\|_{B_{g,2,1}^{s_\theta}}&\leq  C_{\theta,0} \beta^{s_\theta- s_0}\|f\|_{B_{g,2,\infty}^{s_0}}  + C_{\theta,1} \beta^{s_\theta- s_1}\|f\|_{B_{g,2,\infty}^{s_1}} \\          
              &\leq 2C_{\theta}\|f\|_{B_{g,2,\infty}^{s_0}}^{1-\theta}\|f\|_{B_{g,2,\infty}^{s_1}}^{\theta}, 
        \end{align*}
        where $C_\theta = C \cdot \max\{C_{\theta,0}, C_{\theta,1}\}$.
        Finally, it is easy to see that, with constant one, it holds for every $r\in \R, p\in [1, \infty]$
        \begin{align*}  \|f\|_{B_{g,2,\infty}^{r}} \leq  \|f\|_{B_{g,2,p}^{r}} \leq \|f\|_{B_{g,2,1}^{r}}. \end{align*}
        Finally, it is evident by direct inspection that the constant $C_\theta$ is uniformly bounded for $\theta$ varying in a compact subset of $(0,1)$. 
\end{proof}
\begin{proof}[Proof of \autoref{generation_semigroup}]
Linearity as well as closedness and \eqref{inequality_coercivity} trivially follow from the definitions of all the objects involved. In particular, to show \eqref{inequality_coercivity}, the first item in \autoref{hp_correction} is crucially employed. We are left to show that the operator is self-adjoint. By definition, it holds
\begin{align*}
    \langle A u,v\rangle_{X_0}=\langle u,A v\rangle_{X_0}\quad \forall u,v\in \Do(A).
\end{align*}  
This means that $A$ is symmetric.
It remains to show that $\Do(A^*)=\Do(A)$ and $\forall u\in \Do(A),$ $ A^*u=Au$.
By definition \begin{align*}
    \Do(A^*)=\{u \in X_0:\quad & F: \Do(A)\subseteq X_0\rightarrow \mathbb R,\quad F(v)=\langle u,Av\rangle_{X_0}\\ & \text{ has a linear bounded extension on } X_0 \}.
    \end{align*} For each $u\in \Do(A^*),\ F(v)=\langle u,Av\rangle_{X_0}=\langle f_u,v\rangle_{X_0}$ therefore $A^*u=f_u$. In particular, $\forall u\in \Do(A^*)\ \langle u,Av\rangle_{X_0}=\langle A^*u,v\rangle_{X_0}$. Thanks to the fact that $A$ is symmetric, we have $\Do(A)\subseteq \Do(A^*)$. Given now $v\in \Do(A^*)$, let $ f_v=A^*v\in X_0$. Writing $f_v$ in the Fourier series, it is easy to check that there exists a unique $w\in \Do(A)$ such that $Aw=f_v$ as an equality in $X_0$. Therefore, $Aw=f_v=A^*v$ as equality in $X_0$. Similarly, for each $z\in X_0$, there exists a unique $s_z\in \Do(A)$ such that $As_z=z$. Therefore $\langle z,w-v\rangle_{X_0}=\langle As_z,w-v\rangle_{X_0}=\langle s_z,Aw-A^*v\rangle_{X_0}=0$ thanks to the fact that $A$ is symmetric. Since $z$ is arbitrary, then $v=w$ and the claim follows.
\end{proof}

\subsection{Stochastic Maximal Regularity Toolbox}\label{sec:av_theory}
In the body of the manuscript, we have made use of various stochastic maximal regularity results. We recall a simplified version of the main theorems in \cite{agresti_nonlinear_2022} which we have employed in our proofs. We take two Hilbert spaces $X_1 \hookrightarrow X_0$ with dense and compact embedding and a third Hilbert space $\U$ on which we model our noise. We define $X_\theta = [X_0, X_1]_\theta$ as the complex interpolation space and \begin{equation*}
    X^{\operatorname{Tr}}_p\coloneq  (X_0, X_1)_{1-1/p, p}
\end{equation*}
as a real interpolation space.   
We consider some operators $-A: X_1 \rightarrow X_0$ and $B: [0,\infty)\times X_{1}\to \mathscr{L}_2(\U, X_{1/2})$ where $B$ is strongly measurable, and a $\U$-cylindrical Brownian motion $W_t$. We are interested in the maximal regularity properties of the Cauchy problem 
\begin{equation}\label{abstract_pb}
\mathrm{d}u - Au\, \mathrm{d}t = f\,\mathrm{d}t + F(u)\, \mathrm{d}t + B u\, \mathrm{d}W_t,\qquad  u(0)=u_0.
\end{equation}
Let $T>0$. The first definition concerns the maximal regularity of the linear part of the problem 
\begin{definition}
    Let $p>2$. We write $(A, B) \in \mathcal{SMR}^\bullet_p$ if for every progressively measurable $f\in L^p((0,T)\times \Omega, X_0)$, $u_0 \in L^p_{\F_0}(\Omega; X^{\operatorname{Tr}}_p)$ there exists a progressively measurable strong solution $u$ to the linear Cauchy problem 
    \begin{align*}\mathrm{d}u - Au\,\mathrm{d}t = f\,\mathrm{d}t  + B u\,\mathrm{d}W_t,\qquad  u(0)=u_0.\end{align*}
    such that $u \in  L^p(\Omega ; H^{\theta, p}(0,T; X_{1-\theta}))$ for every $\theta \in [0, 1/2)$ and for all stopping times $\tau : \Omega \to [0,T]$ it holds 
    \begin{align*}\|u\|_{L^p( \Omega; H^{\theta, p}(0,\tau; X_{1-\theta}))} \lesssim \|f\|_{L^p((0,\tau) \times \Omega ; X_0)} + \|u_0\|_{L^p(\Omega; X^{\operatorname{Tr}}_p)}\end{align*}
    and in particular 
    \begin{equation}\label{smr-inequality}
        \EE\Big[\sup_{t\in[0, T]} \|u_t\|^p_{X^{\operatorname{Tr}}_p}\Big]  +\EE\left[\int_0^T\|u_t\|^p_{X_1}\,\mathrm{d}t \right]\lesssim\EE\big[\|u_0\|^p_{ X^{\operatorname{Tr}}_p}\big]+ \EE\left[\int_0^T \|f_t\|^p_{X_0}\,\mathrm{d}t\right].
    \end{equation}
\end{definition}

\begin{hypothesis}\label{SMR_nonlinear_hp}
        Let $X_0$ be a separable Hilbert space and suppose that there is a $\lambda_0 \geq 0$ such that $\lambda_0 - A$ is sectorial on $X_0$ and set $X_1 = \Do(-A)$. Let $p> 2$ and assume that for some integer $m\geq 1$ the mapping 
    $F:X_1 \rightarrow X_0$
    satisfies 
    \begin{align}
        \|F(u)- F(v)\|_{X_0} \lesssim \sum_{j=1}^m(1 + \|u\|_{X_{\beta_j}}^{\rho_j} + \|v\|_{X_{\beta_j}}^{\rho_j})\|u-v\|_{X_{\theta_j}}
    \end{align}
	for all $u,v\in X_1$, where $1-\frac{1}{p}<\theta_j\leq\beta_j<1$ and $\rho_j>0$ satisfy
    \begin{align}\label{criticality_cond}
        \rho_j \left(\beta_j - 1 + \frac{1}{p}\right) + \theta_j \leq 1\quad \text{for each }j\in \{1,\dots,m\}.\quad \text{{\normalfont{[(Sub)critical condition]}}}
    \end{align}
\end{hypothesis}

\begin{theorem}
\label{SMR_maximal_thm}
    Let $p> 2$ and assume that \autoref{SMR_nonlinear_hp} is satisfied and that $(A,B)\in \mathcal{SMR}^\bullet_p$. Assume also that $u_0 \in L^\infty_{\F_0}(\Omega; X^{\operatorname{Tr}}_p)$. Then, for every progressively measurable $f\in L^p(\Omega\times (0,T); X_0)$ there exists a unique maximal $p$-solution $(u, \sigma)$ to \eqref{abstract_pb} with $\sigma>0$ almost surely and for every stopping time $\tau < \sigma$
    \begin{align*}
    u\in H^{\vartheta,p}_{\mathrm{loc}}
    ([0,\sigma);X_{1-\vartheta})
    \cap C([0,\sigma);X^{\operatorname{Tr}}_p)
    \qquad\text{for every }\vartheta\in[0,1/2)\quad a.s.
    \end{align*}
\end{theorem}
The following result concerns the blow-up of maximal $L^p$-solutions. The proof can be found, e.g., in \cite[Theorem 4.10 (2)-(3)]{agresti_nonlinear_2022}.

\begin{theorem}\label{SMR_blowup+instreg}
    Suppose that \eqref{SMR_nonlinear_hp} and the hypothesis of \autoref{SMR_maximal_thm} hold and let $(u, \sigma)$ be a maximal $p$-solution of \eqref{abstract_pb}, then
    \begin{itemize}
        \item $\displaystyle{\PP\Big( \sigma< \infty, \sup_{t\in [0, \sigma)}\|u_t\|_{X^{\operatorname{Tr}}_p} + \int_0^\sigma \|u_t\|^p_{X_1} \mathrm{d}t < \infty \Big) =0}$;
        \item $\displaystyle{\PP\Big( \sigma< \infty, \sup_{t\in [0, \sigma)}\|u_t\|_{X^{\operatorname{Tr}}_p} <\infty\Big) =0}$ if \eqref{criticality_cond} holds with strict inequality for all $j$.
    \end{itemize}
\end{theorem}

\begin{proposition}\label{SMR_linear_prop}
Suppose that $\|(\sigma_k)_{k\in I}\|_{\ell^2(L^\infty(\T^3))}\leq N$ for some $N\geq 1$.  
    Let $A$ be as defined in \autoref{Defintition_operator} with ground space $H^{-2}_g$, i.e., $
        A: \Do(A)=H^0_g \rightarrow H^{-2}_g
    $
   and 
   \begin{align*}
B: [0,T]\times  H^0_g \rightarrow \mathscr{L}_2(\mathcal{U},H^{-1}_g),
   \end{align*}
defined as $B(t)\omega= (\LL_{\sigma_k}\omega)_{k\in I}$. Then 
   $(A, B)\in \mathcal{SMR}^\bullet_p$, for every $p> 2$. Moreover, it follows that for every $\nu\geq 0$ also $(A+\nu\Delta, B)\in \mathcal{SMR}_p^\bullet$ with the same ground space, where the hidden constant in \eqref{smr-inequality} might depend on $\nu$. 
   \end{proposition}
   \begin{proof}
   The proof is done by following the proof in \cite[Theorem A.1]{agresti_anomalous_2024}, which is an application of the perturbation method developed in \cite{agresti_stochastic_2024}, provided that $-A$ has a bounded $H^\infty$ calculus of angle zero, and the following lemma is applied pointwise in time. 
 \end{proof}

\begin{lemma}\label{SMR_perturbation_estimate}
Suppose that $\|(\sigma_k)_{k\in I}\|_{\ell^2(L^\infty(\T^3))}\leq N$. Then for every $\eps>0$ there exists $C_{\eps,N}$ depending on $\eps,\ N$ and $g$ such that it holds 
     \begin{equation}\label{noise_est}
        \|\Delta \omega\|_{H^{-2}_g} + \|B \omega\|_{\L_2( \ell^2, H^{-1})} \leq \eps \|\omega\|_{H^{0}_g} + C_{\eps,N} \|\omega\|_{H^{-1}}.
    \end{equation}
\end{lemma}
\begin{proof}
    We have 
     \begin{align*}\|\Delta \omega\|_{H^{-2}_g} = \|\omega\|_{H^{0}_{g^{-1}}} \leq \eps \|\omega\|_{H^{0}_g} + C_\eps\|\omega\|_{H^{-1}}\end{align*}
   where the last inequality follows by noticing that for any $M>1$, since $g$ is increasing, it holds
    \begin{align*}\|\omega\|^2_{H^{0}_{g^{-1}}}= \sum_{k}g^{-1}(|k|)|\omega_k|^2 \leq {g^{-1}(1)}M^2\sum_{|k|\leq M}|k|^{-2}|\omega_k|^2 + g^{-2}(M)\sum_{|k|>M}g^{}(|k|)|\omega_k|^2.\end{align*}
    Thus, the estimate follows by choosing $M$ large enough.
   By a similar argument and thanks to \autoref{lem:noise_estimate}, for every $M\geq 1$ we have
    \begin{align}\label{ito_corr_estimate}
        \|B\omega \|^2_{\L_2(\ell^2, H^{-1})} = \sum_{k\in I}\|\LL_{\sigma_k} \omega\|_{H^{-1}}^2 \leq C N^2\|\omega\|_{H^0}^2 \leq CN^2 g^{-1}(M) \|\omega\|^2_{H^{0}_g} + CN^2 M^2 \|\omega \|^2_{H^{-1}}. 
    \end{align}
    So, we can choose $M$ depending on $C$ and $g$ to satisfy the desired inequality. 
\end{proof}
\subsection{Proof of \autoref{lem:high_mode_energy_defect}}
\label{app_techical_lemma}
By density we can assume that $v$ has only a finite number of nontrivial Fourier modes.
Recall from \autoref{hyp_noise} that 
\begin{equation}
\label{eq:noise_assumption_hyp_lemma}
 \sum_{m\in2\pi\Z_0^3}|\theta_m^n|^2=\frac32,\qquad
 \sum_{k=(m,j)\in I}|\theta_m^n|^2=3,\qquad
 \theta_m^n=0\ \  \text{outside }\ n\leq |m|\leq2n.
\end{equation}
Next, note that, for a fixed $k=(m,j)\in I$, the identity $\LL_\sigma v=-\nabla\times(\sigma\times v)$ yields
\begin{align*}
 \LL_{\sigma_{m,j}^n}v(x)
 =-i\theta_m^n\sum_{\lambda\in2\pi\Z_0^3}
 (\lambda+m)\times(a_{m,j}\times\widehat v(\lambda))e^{i(\lambda+m)\cdot x}.
\end{align*}
Parseval's identity and \autoref{def_function_spaces} therefore imply
\begin{align}
\label{eq:identity_LL_proof_lemma}
 \|\LL_{\sigma_{m,j}^n}v\|_{H^{-1}}^2
 &=
 |\theta_m^n|^2
 \sum_{\substack{\lambda\in2\pi\Z_0^3\\ \lambda+m\neq0}}
 \frac{|(\lambda+m)\times(a_{m,j}\times\widehat v(\lambda))|^2}{|\lambda+m|^2}.
\end{align}
For notational convenience, below, for $z\in \C^3$, we set
\begin{align*}
 \Phi_n(\lambda,z)\coloneq
 \sum_{\substack{k=(m,j)\in I\\ \lambda+m\neq0}}
 |\theta_m^n|^2
 \frac{|(\lambda+m)\times(a_{m,j}\times z)|^2}{|\lambda+m|^2}.
\end{align*}
Next, note that symmetry and \eqref{eq:noise_assumption_hyp_lemma} imply $ \sum_{m\in2\pi\Z_0^3}|\theta_m^n|^2
\frac{m\otimes m}{|m|^2} =\frac12\id_{\R^3}$ for all $n$ and 
\begin{align*}
0\leq \Phi_n (\lambda , z)\leq 3 |z|^2. 
\end{align*}
Moreover, as $m\cdot a_{m,j}=0$, we have 
$
 \frac m{|m|}\times(a_{m,j}\times z)
 =a_{m,j}\left(\frac m{|m|}\cdot z\right),
 $ for all  $ z\in\C^3$.
As $a_{m,j}$ are unit vectors, we have 
\begin{align}
 \sum_{k=(m,j)\in I}|\theta_m^n|^2
 \left|\frac{ m\times(a_{m,j}\times z)}{|m|}\right|^2
 =2z^*\left(\frac12\mathrm{Id}_{\R^3}\right)z
 =|z|^2.
 \label{eq:high_mode_isotropic_contribution}
\end{align}
Comparing the previous displayed formula with \eqref{eq:identity_LL_proof_lemma}, we need to estimate the error obtained by replacing $m$ with $m+\lambda$. To this end, note that, for each $|m| \geq n$ and $\lambda$ such that $\lambda + m \neq 0$, we have   
\begin{align*}
 \left|\frac{\lambda+m}{|\lambda+m|}-\frac m{|m|}\right|
 &\leq
 \left|\frac{\lambda+m}{|\lambda+m|}-\frac{\lambda+m}{|m|}\right|
 +\frac{|\lambda|}{|m|}\\
 &=\frac{\left||m|-|\lambda+m|\right|+|\lambda|}{|m|}
 \leq\frac{2|\lambda|}{|m|}
 \leq\frac{2|\lambda|}{n}.
\end{align*}
Moreover, for real unit vectors $\xi,\eta$ and $b\in\C^3$,
\begin{align*}
 \left||\xi\times b|^2-|\eta\times b|^2\right|
 \leq
 (|\xi\times b|+|\eta\times b|)\,|(\xi-\eta)\times b|
 \leq2|\xi-\eta|\,|b|^2.
\end{align*}
Applying this with
$\xi=(\lambda+m)/|\lambda+m|$, $\eta=m/|m|$ and $b=a_{m,j}\times z$,
and using $|a_{m,j}\times z|\leq|z|$, gives
\begin{align*}
 \left|
 \left|\frac{\lambda+m}{|\lambda+m|}\times(a_{m,j}\times z)\right|^2
 -
 \left|\frac m{|m|}\times(a_{m,j}\times z)\right|^2
 \right|
 \leq\frac{4|\lambda|}{n}|z|^2.
\end{align*}
For $0<|\lambda|\leq\Lambda<n/2$, we have $\lambda+m\neq0$ whenever $\theta_m^n\neq0$.
Therefore, the previous identities and \eqref{eq:noise_assumption_hyp_lemma} imply for $0<|\lambda|\leq\Lambda<n/2$
\begin{align}
 \left|\Phi_n(\lambda,z)-|z|^2\right|
 &\leq
 \sum_{\substack{k=(m,j)\in I\\ \theta_m^n\neq0}}
 |\theta_m^n|^2
 \left|
 \left|\frac{\lambda+m}{|\lambda+m|}\times(a_{m,j}\times z)\right|^2
 -
 \left|\frac m{|m|}\times(a_{m,j}\times z)\right|^2
 \right|
 \notag\\
 &\leq\frac{4|\lambda|}{n}|z|^2
       \sum_{k=(m,j)\in I}|\theta_m^n|^2
 =\frac{12|\lambda|}{n}|z|^2.
 \label{eq:high_mode_low_frequency_error}
\end{align}
We are now ready to prove the first estimate \eqref{eq:high_mode_energy_defect}. Indeed, by Parseval's identity and \eqref{eq:identity_LL_proof_lemma}, for all $2\pi \leq \Lambda <n/2$, it holds that 
\begin{align*}
 \Big|
 \sum_{k\in I}\|\LL_{\sigma_k^n}v\|_{H^{-1}}^2-\|v\|_{L^2}^2 \Big|
 &\leq
 \sum_{\substack{\lambda\in  2\pi\Z_0^3\\ 0<|\lambda|\leq \Lambda}}
 \left|\Phi_n(\lambda,\widehat v(\lambda))-|\widehat v(\lambda)|^2\right|
 +
 \sum_{\substack{\lambda\in 2\pi\Z_0^3\\ |\lambda|>\Lambda}}
 \left|\Phi_n(\lambda,\widehat v(\lambda))-|\widehat v(\lambda)|^2\right|\\
 &\leq
 \frac{12}{n}\sum_{\substack{\lambda\in 2\pi\Z_0^3\\ 0<|\lambda|\leq \Lambda}}|\lambda|\,|\widehat v(\lambda)|^2
 +4\sum_{\substack{\lambda\in 2\pi\Z_0^3\\ |\lambda|>\Lambda}}|\widehat v(\lambda)|^2\\
 &\leq
 \frac{12\Lambda}{n}\sum_{\substack{\lambda\in  2\pi\Z_0^3\\ 0<|\lambda|\leq \Lambda}}|\widehat v(\lambda)|^2
 +4\sum_{\substack{\lambda\in 2\pi\Z_0^3\\ |\lambda|>\Lambda}}|\widehat v(\lambda)|^2\\
 &\lesssim \Big(\frac{\Lambda}{n}+ \frac{1}{g(\Lambda)}\Big)\|v\|_{H^0_g}^2,
\end{align*}
where in the last step we used the monotonicity of $g$ and the definition of the space $H^0_g$. 

Finally, we prove \eqref{eq:high_mode_martingale_defect}. 
For $2\pi\leq\Lambda<n/2$, let $P_\Lambda$ be the Fourier projection on the modes $\leq \Lambda$. 
If $\theta_m^n\neq0$ and $0<|\lambda|\leq\Lambda<n/2$, then $|m|\geq n$ and
\begin{align*}
 |\lambda+m|\geq |m|-|\lambda|\geq n-\Lambda>\Lambda.
\end{align*}
From \eqref{eq:noise_assumption_hyp_lemma} and $\sigma_k^n = \theta^n_m a_{m,j} e^{im\cdot x}$ where $k=(m,j)$, it follows that 
$P_\Lambda\LL_{\sigma_k^n}P_\Lambda v=0$ for all $k\in I$.
Since $P_\Lambda$ is an orthogonal projection in $H^{-1}$, we obtain, for all $k\in I$,
\begin{align*}
 \langle\LL_{\sigma_k^n}P_\Lambda v,P_\Lambda v\rangle_{H^{-1}}=0.
\end{align*}
Therefore
\begin{align}
 \langle\LL_{\sigma_k^n}v,v\rangle_{H^{-1}}
 &=
 \langle\LL_{\sigma_k^n}P_\Lambda v,(\id-P_\Lambda)v\rangle_{H^{-1}}+
 \langle\LL_{\sigma_k^n}(\id-P_\Lambda)v,v\rangle_{H^{-1}}.
 \label{eq:high_mode_pairing_split}
\end{align}
By \eqref{eq:high_mode_pairing_split}, Cauchy--Schwarz inequality and
\autoref{lem:noise_estimate}, we consequently have
\begin{align*}
 \sum_{k\in I}
 |\langle\LL_{\sigma_k^n}v,v\rangle_{H^{-1}}|^2&\leq
 2\|(\id-P_\Lambda)v\|_{H^{-1}}^2
 \sum_{k\in I}\|\LL_{\sigma_k^n}P_\Lambda v\|_{H^{-1}}^2+
 2\|v\|_{H^{-1}}^2
 \sum_{k\in I}\|\LL_{\sigma_k^n}(\id-P_\Lambda)v\|_{H^{-1}}^2\\
 &\lesssim
 \|(\id-P_\Lambda)v\|_{H^{-1}}^2\|P_\Lambda v\|_{L^2}^2
 +\|v\|_{H^{-1}}^2\|(\id-P_\Lambda)v\|_{L^2}^2.
\end{align*}
Next, as $\|P_\Lambda v\|_{L^2}^2 \leq\Lambda^2\|v\|_{H^{-1}}^2$ and $ \|(\id-P_\Lambda)v\|_{H^{-1}}^2 \leq\Lambda^{-2}\|(\id-P_\Lambda)v\|_{L^2}^2$, we have 
\begin{align*}\|(\id-P_\Lambda)v\|_{H^{-1}}^2\|P_\Lambda v\|_{L^2}^2\leq \|v\|_{H^{-1}}^2\|(\id-P_\Lambda)v\|_{L^2}^2
\end{align*} 
and the previous implies
\begin{align*}
 \sum_{k\in I}
 |\langle\LL_{\sigma_k^n}v,v\rangle_{H^{-1}}|^2
 &\lesssim \|v\|_{H^{-1}}^2\|(\id-P_\Lambda)v\|_{L^2}^2\\ &
 =\|v\|_{H^{-1}}^2
 \sum_{|\lambda|>\Lambda}|\widehat v(\lambda)|^2\lesssim 
 \frac{1}{g(\Lambda)}\|v\|_{H^{-1}}^2\|v\|_{H_g^0}^2,
\end{align*}
where, as above, the last step follows from the definition of $H_g^0$.
\end{appendix}

\end{document}